\pdfoutput=1

\documentclass[11pt,a4paper]{amsart}

\calclayout

\usepackage{mathtools}
\usepackage{amssymb}
\usepackage{hyperref}
\usepackage{subcaption}
\usepackage{listings}
\usepackage{threeparttable}
\usepackage{booktabs}

\hypersetup{bookmarksdepth=subsection}
\numberwithin{equation}{section}
\newtheorem{theorem}{Theorem}[section]
\newtheorem{corollary}[theorem]{Corollary}
\newtheorem{conjecture}[theorem]{Conjecture}
\newtheorem{proposition}[theorem]{Proposition}
\newtheorem{gfgp}{GFG Property}
\theoremstyle{definition}
\newtheorem{definition}[theorem]{Definition}
\theoremstyle{remark}
\newtheorem{remark}[theorem]{Remark}

\newcommand{\ccol}[1]{\multicolumn{1}{c}{#1}}
\newcommand{\ud}{\underline}

\begin{document}

\title{Exceptional autonomous components of Goldbach factorization graphs}
\author{Andrzej Bo\.zek}

\address{Department of Computer and Control Engineering, Rzesz\'ow University of Technology, Powsta\'nc\'ow Warszawy 12, 35-959 Rzesz\'ow, Poland}
\email{abozek@prz.edu.pl}

\begin{abstract}
We introduce a concept of a Goldbach factorization graph (GFG) $F_n$, which can be constructed for each even integer $n$ greater than 2. We prove that, if $n$ does not satisfy the binary Goldbach conjecture (BGC), then $F_n$ contains a special source strongly connected component (exceptional autonomous component, EAC). We analyse existence and properties of EACs using deductive and computational approaches. In~particular, we prove that there exists exactly one EAC induced by two vertices. Using computer-aided search, we show that for $n \leq 10^8$ there are 6 EACs, each inside a different GFG, and they are located at the relative beginning of the checked range, namely, for $n\in\{128,1718,1862,1928,2200,6142\}$. Using classic graph algorithms, the constraint programming method, and metaheuristic approaches, we have prepared a repository of drawings and some selected properties of the found EACs and GFGs which contain them. The concept of EAC relates to the BGC, but more generally, it represents interesting self-conjugation of prime numbers under a relation which combines addition and multiplication.
\end{abstract}


\maketitle

\section{Introduction}
\label{sec:Introduction}
We will use the standard notations: $\mathbb{P}$ for the set of prime numbers, $\pi$ for the prime-counting function, $\omega$ for the distinct prime factor-counting function, and $\mathbb{N}_0=\{0,1,\ldots\}$, as well as $\mathbb{N}_1=\{1,2,\ldots\}$ for subsets of integers. Let $\mathbb{N}_\mathrm{G}=\left\{4+2i:i \in \mathbb{N}_0\right\}$ and
\[
	g:\mathbb{N}_\mathrm{G} \mapsto \mathbb{N}_0 \quad \textrm{such that } g(n)=\#\{(p,q)\in\mathbb{P}^2 : p+q=n,p \leq q\}
\]
represent the number of so called Goldbach partitions of $n$. The well known binary (strong/even) Goldbach conjecture (BGC) states that $g(n)>0$ for every $n \in \mathbb{N}_\mathrm{G}$.

Estermann proved, about 1938, that almost all even positive integers are sums of two primes%
~\cite{Estermann1938}. In mid-seventies Chen proved that \emph{every sufficiently large even number can be written as the sum of either two primes, or a prime and a semiprime}%
~\cite{Chen1973}. In 2015, Yamada estimated an explicit upper bound on \textit{the sufficiently large even number} from Chen's theorem, equal to $\exp\left(\exp(36)\right)$%
~\cite{Yamada2015}.
It is easy to show that the BGC implies the following weaker statement: \emph{every odd number greater than five can be expressed as the sum of three primes}, called ternary (weak/odd) Goldbach conjecture (TGC). It was proved first that the TGC holds under the generalized Riemann hypothesis%
~\cite{Deshouillers1997}, and then, a complete proof of this conjecture was published in 2013%
~\cite{Helfgott2013}, which is widely accepted by mathematicians. The proof is based on the circle method, introduced by Hardy, Littlewood, and  Ramanujan%
~\cite{Hardy1918,Hardy1924}, which is probably the most successful and most promising approach for the Goldbach problem. Some research was conducted on the problem of an upper bound on $g(n)$. Deshouillers~et~al. proved that this upper bound is less than the trivial value $\pi\left(n-2\right)-\pi\left(\frac{1}{2}n-\frac{1}{2}\right)$ for each $n>210$%
~\cite{Deshouillers1993}.

Empirical verification is a separate thread in the research on the BGC. Significant empirical results became possible in the era of computers. The first well-documented computer-aided verification originates from 1964 and it covers the integers up to $3.3 \cdot 10^7$%
~\cite{Shen1964}. The newest commonly known result, for the integers up to $4\cdot 10^{18}$, obtained Oliveira e Silva et al. in 2012%
~\cite{Silva2014}. Regardless of sophistication of the used algorithms, the existing results typically concern computing of exact or estimated values of $g(n)$, which are sometimes presented in a diversified form, e.g. the Goldbach comet~\cite{Fliegel1989}. In this work, we propose essentially new concepts related to the BGC, which can be studied using deductive and computational methods. In particular, this concepts involve a directed graph for each $n \in \mathbb{N}_\mathrm{G}$, having $\pi(n-2)$ vertices, in which special kinds of source strongly connected components are sought and (in a deeper research) analysed. It indicates that the computational cost of this research is higher than the computing of $g(n)$, hence, the proposed concepts are an interesting challenge for modern computer science and computation theory.

To introduce the new concepts, let us consider the process of verification of the BGC for an arbitrary $n \in \mathbb{N}_\mathrm{G}$, treating it as an algorithmic task. If $n$ satisfies the BGC, some primes from the set $\mathbb{V}_n=[2,n-2] \cap \mathbb{P}$ combine into Goldbach partitions, i.e., there exists at least one $p\in\mathbb{V}_n$ such that $n-p\in\mathbb{V}_n$. It is obvious that $p\in\mathbb{V}_n$ such that $p \mid n$ and $n / p > 2$ does not belong to a Goldbach partition. However, for better generalization of the introduced concepts, we will assume that all $p\in\mathbb{V}_n$ are checked against belonging to a Goldbach partition. This checking can be computationally represented by evaluation of two functions. First, the function
\[
f_n^\mathrm{G}: \mathbb{V}_n \mapsto \mathbb{N}_{[2,n-2]}, \;\textrm{where}\; \mathbb{N}_{[2,n-2]}=[2,n-2]\cap\mathbb{N}, 
\;\textrm{such that}\; f_n^\mathrm{A}(x)=n-x,
\]
which we can refer to as a \textit{Goldbach complement function}, is evaluated for a checked prime $p\in\mathbb{V}_n$. Then, the result of $f_n^\mathrm{G}$ is evaluated with the use of the \textit{factorization function} $f_n^\mathrm{F}: \mathbb{N}_{[2,n-2]} \mapsto 2^{\mathbb{V}_n \times \mathbb{N}_1}$ such that
\[
f_n^\mathrm{F}(x)=\bigcup\limits_{i=1}^{\omega(x)}\left\{(b_i,e_i)\right\},
\;\textrm{where}\; \prod\limits_{i=1}^{\omega(x)} b_i^{e_i}=x,
\;\Big(\forall{i\in\{2,3,\ldots,\omega(x)\}}\Big)b_{i-1}<b_i.
\]
Let us notice, that the fundamental theorem of arithmetic%
~\cite{Gauss1986} asserts that $f_n^\mathrm{F}$ is a well defined function. The checked prime $p$ belongs to a Goldbach partition if and only if the result of $f_n^\mathrm{F}$ is a single element set of the form $\left\{(b,1)\right\}$. We will not be interested in the intermediate result of $f_n^\mathrm{G}$ evaluated by $f_n^\mathrm{F}$, therefore, we can compose these functions into one \textit{Goldbach factorization function} $f_n=f_n^\mathrm{G}\circ f_n^\mathrm{F}$, such that $f_n: \mathbb{V}_n \mapsto 2^{\mathbb{V}_n \times \mathbb{N}_1}$. A prime $b \in \mathbb{V}_n$ will be called a \textit{Goldbach factor} of $p$ if and only if $(b,e) \in f_n(p)$ for some $e \in \mathbb{N}_1$, assuming that $n \in \mathbb{N}_\mathrm{G}$ is known from the context.

To obtain a better insight into the Goldbach factorization relation represented by $f_n$, we will use a graph $F_n=\left(V_n,A_n,w_n\right)$, where $V_n=\mathbb{V}_n$, $A_n \subseteq V_n^2$, and $w_n: A_n \mapsto \mathbb{N}_1$. The arcs of $F_n$ and their weights are defined as follows:
\[
\forall (p,q) \in V_n^2 \;\big[(p,q) \in A_n \iff \left(\exists e \in \mathbb{N}_1\right)(p,e)\in f_n(q)\big],
\]
\[
\forall (p,q) \in A_n \;\big[w_n\big((p,q)\big)=e: (p,e)\in f_n(q)\big].
\]

\begin{figure}
	\centering
	\includegraphics[width=0.8\textwidth,trim={51mm 147mm 43mm 44mm},clip]{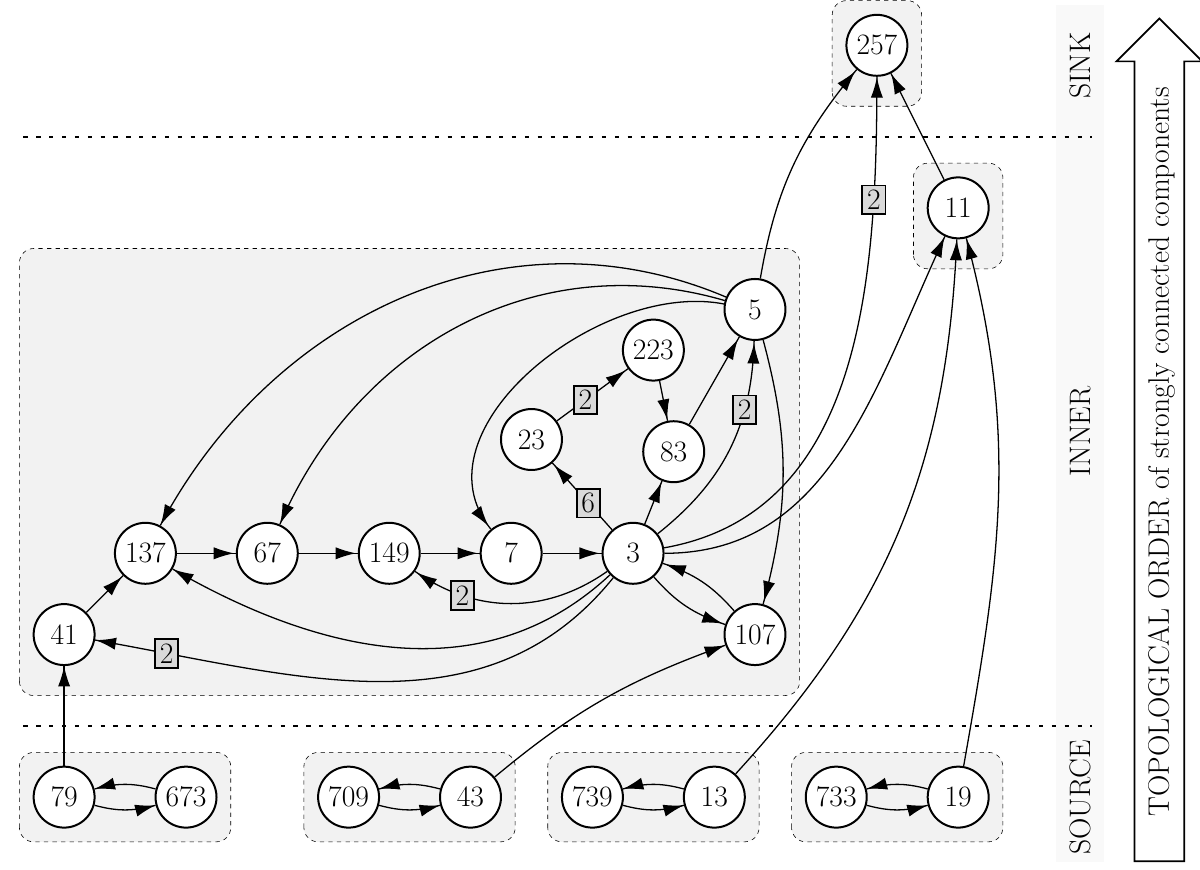}
	\caption{The subgraph $F_{752}^{(257)}$ of $F_{752}$}
	\label{fig:GFG752}
\end{figure}

Following the previous naming convention, $F_n$ will be referred to as a~\textit{Goldbach factorization graph} (GFG). We will consider a part of an exemplary GFG to analyse its properties. Let us choose $n=752$ and construct the subgraph $F_{752}^{(257)}$ of $F_{752}$ induced by the vertices reachable from the vertex 257, going along the reverse direction of arcs. The graph $F_{752}^{(257)}$ is presented in Figure~\ref{fig:GFG752}. If an arc weight is equal to 1, the respective label is omitted in the figure for legibility, the same rule is followed in the remaining figures of GFGs. We have the following observations:
\begin{enumerate}
	\item GFGs have, as a rule, a non-trivial topological structure. The considered small subgraph $F_{752}^{(257)}$ has, e.g., many cycles and a few strongly connected components (SCCs)~\cite{Cormen2009} of different sizes. It is easy to notice, that this is an emergent property of the composition $f_n^\mathrm{G}\circ f_n^\mathrm{F}$. Indeed, graph representations of the functions $f_n^\mathrm{G}$ and $f_n^\mathrm{F}$, defined analogously as $F_n$, would be topologically simple, at least in this sense, that they would be acyclic (except for loops). The composition $f_n^\mathrm{G}\circ f_n^\mathrm{F}$ results in circuitous topological properties, because this composition imposes that each predecessor of $p \in V_n$, being a Goldbach factor of $p$, can be less then, equal to, or greater than $p$, and this implies intricate precedence relations between vertices.
	\item The traversing of a GFG along the reverse direction of arcs is naturally supported by the definition of $f_n$, because $f_n$ returns a set describing predecessors of a given vertex.
	\item In general, not all vertices of a GFG are reachable in the reverted traversing from an arbitrary vertex. For example, $F_{752}$ and $F_{752}^{(257)}$ have 132 and 21 vertices, respectively. If a subset $V_n^{(p)} \subset V_n$ contains all the vertices of $F_n$ reachable in the reverted traversing from $p$, then there are no arcs from any vertex $a \in V_n \setminus V_n^{(p)}$ to any vertex $b \in V_n^{(p)}$. In other words, the numbers in $V_n \setminus V_n^{(p)}$ are not Goldbach factors of any number in $V_n^{(p)}$ and the subset $V_n^{(p)}$ is closed under the relation of Goldbach factorization in this sense, that there exists the function $f_n^{(p)}: V_n^{(p)} \mapsto 2^{V_n^{(p)} \times \mathbb{N}_1}$ such that $\left(\forall v \in V_n^{(p)}\right)f_n^{(p)}(v)=f_n(v)$.
\end{enumerate}
The graph $F_{752}^{(257)}$ is not a minimal (in the sense of inclusion) subgraph of $F_{752}$ closed under the Goldbach factorization. It contains the subgraphs induced by $\{13,739\}$, $\{19,733\}$, $\{43,709\}$, and $\{79,673\}$, which have this property of minimality. Such minimal subgraphs have a specific role in the structure of a GFG. They provide Goldbach factors for the remaining part of the GFG, while they do not rely on other factors themselves. One can say, that these minimal subgraphs represent sets of the primes which are strictly self-conjugated under the relation of Goldbach factorization. To emphasis this property, such subgraphs will be referred to as \textit{autonomous components} of GFGs (see Definition~\ref{def:AC}). It is easy to prove, that autonomous components are equivalent to source SCCs of GFGs (see GFG Propoerty~\ref{gfgp:AutCmp}).

The reverse-directed traversing of $F_n$ represents also a reasonable verification algorithm of satisfaction of the BGC by a number $n$. One starts from some $p \in V_n$ and checks the predecessors of $p$. The set of predecessors defines the prime $p$ ''in term'' of its Goldbach factors. These predecessors and related arc weighs determine whether $p$ belongs to a Goldbach partition. If $p$ does not belong to a Goldbach partition, the same verification has to be recursively done for its predecessors, and so on. Finally, the traversing has to reach an autonomous component. In the considered example of $F_{752}^{(257)}$, all the autonomous components represent Goldbach partitions. It is easy to prove, that each Goldbach partition induces an autonomous component in a GFG (see GFG Properties~\ref{gfgp:AutCmp}~and~\ref{gfgp:GFGpartition}), thus, this type of an autonomous component (\textit{Goldbach autonomous component, GAC}) is common. Let us analyse a few more detailed properties of Goldbach factorization graphs:
\begin{enumerate}
	\item According to the definition of $f_n$ and the mapping rule between $f_n$ and $F_n$, each vertex in $F_n$ has a predecessor.
	\item A graph $F_n$ may have loops, i.e. arcs of the form $(v,v) \in A_n$ (see GFG Property~\ref{gfgp:Loop1}).
	\item If a vertex of a GFG has a loop (i.e., it is a successor of itself), it has no other successors (see GFG Property~\ref{gfgp:Loop2}).
	\item Each GFG has vertices without loops, except for $F_4$ and $F_6$ (see GFG Property~\ref{gfgp:AllGFGloops}).
	\item The previous points in conjunction imply that each GFG (except for $F_4$ and $F_6$) has at least one autonomous component induced by more than one vertex.
\end{enumerate}
We notice that, if there exists $n \in \mathbb{N}_\mathrm{G}$ which does not satisfy the BGC, then $F_n$ contains an autonomous component induced by at least two vertices which is not a GAC (see Theorem~\ref{th:EACGoldbach}). This is the main observation of this work, and the presented research is targeted into analysis of the existence of such specific autonomous components, which we will refer to as \textit{exceptional autonomous components} (EAC). In particular, in this work we try to answer the following questions:
\begin{enumerate}
	\item Is it possible to obtain rigorous statements related to the EACs existence, based on the existing mathematical knowledge?
	\item How to search for EACs with the use of computer-aided techniques and what are the results of such experimental research?
	\item Can we characterize the cardinality of the set of EACs: empty, finite, infinite?
	\item Are there some general properties of EACs?
	\item If EACs exist, is it practical (regarding their number and sizes) to describe each of them individually?
\end{enumerate}

In Section~\ref{sec:Preliminaries}, basic definitions are given and properties of the defined objects are introduced. In Section~\ref{sec:TwinEACs}, twin EACs, i.e. the EACs induced by two vertices, are examined. Section~\ref{sec:EACSearch} covers design of a computer algorithm which searches for EACs. The review of selected properties of the found EACs, together with justification of correctness of the obtained results, are included in Section~\ref{sec:EACReview}. In Section~\ref{sec:Remarks}, the above stated questions are answered and final remarks are given.

\section{Preliminaries}
\label{sec:Preliminaries}
The concepts of the Goldbach factorization graph and its autonomous components have been extensively explained in Introduction, but the following brief definitions are sufficient for practical use.

\begin{definition}
Given $n \in \mathbb{N}_\mathrm{G}$, a \textit{Goldbach factorization graph} is the directed weighted graph $F_n=\left(V_n,A_n,w_n\right)$,\\where $V_n=[2,n-2] \cap \mathbb{P}$, $A_n=\left\{(s,t) \in V_n^2 : s \mid (n-t)\right\}$,\\and $w_n: A_n \mapsto \mathbb{N}_1$ such that $w_n\big((s,t)\big)=\mathop{\mathrm{max}}_{\,e \in \mathbb{N}_1}\left[(s^e \mid (n-t)\right]$.
\label{def:GFG}
\end{definition}

\begin{definition}
An \textit{autonomous component} of a Goldbach factorization graph $F_n$ is a minimal (in the sense of inclusion) subgraph of $F_n$ induced by a subset of vertices $U \subseteq V_n$ such that $(s,t) \notin A_n$ for each $s \in V_n \setminus U$ and $t \in U$.
\label{def:AC}
\end{definition}

We will prove a few useful properties related to GFGs and their autonomous components.

\begin{gfgp}
A subgraph of a Goldbach factorization graph $F_n$ is an autonomous component if and only if it is a source strongly connected component of $F_n$.
\label{gfgp:AutCmp}
\end{gfgp}
\begin{proof}
Assume that a subgraph $S_X$ of $F_n$ induced by the vertices from $X \subset V_n$ is a source SCC. $S_X$ has no predecessors in the condensation graph of $F_n$, thus, if $s \subset V_n \setminus X$ and $t \in X$, then $(s,t) \notin A_n$. Let $S_Y$ be an arbitrary subgraph of $S_X$ induced by $Y \subset X$. The property of strong connectivity asserts that there exists $(s,t) \in A_n$ such that $s \in X \setminus Y$ and $t \in Y$. According to Definition~\ref{def:AC}, it follows that $S_X$ is an autonomous component and, by the minimality condition, no supergraph of $S_X$ is an autonomous component of $F_n$. Consider $Z \subset V_n$ inducing a subgraph $S_Z$ of $F_n$ which does not include any complete source SCC. There exists $(s,t) \in A_n$ such that $s \in V_n \setminus Z$ and $t \in Z$, because each vertex in a directed graph is reachable from all vertices of some source SCC. Therefore, $S_Z$ is not an autonomous component of $F_n$.
\end{proof}

\begin{gfgp}
A vertex $v \in V_n$ in a graph $F_n$ has a loop, i.e. an arc $(v,v)$, if and only if $v \mid n$.
\label{gfgp:Loop1}
\end{gfgp}
\begin{proof}
If $v \mid n$, then $v \mid (n-v)$, hence, $(v,v) \in A_n$, otherwise, if $v \nmid n$, then $v \nmid (n-v)$, therefore, $(v,v) \notin A_n$, according to Definition~\ref{def:GFG}.
\end{proof}

\begin{gfgp}
If a vertex $s \in V_n$ in a graph $F_n$ has a loop, then there does not exist $(s,t) \in A_n$ such that $s \neq t$.
\label{gfgp:Loop2}
\end{gfgp}
\begin{proof}
According to GFG Property~\ref{gfgp:Loop1}, if a vertex $s \in V_n$ has a loop, then $s \mid n$. Therefore, there does not exist a vertex $t \in V_n$, $t \neq s$, such that $s \mid (n-t)$, because $s,t \in \mathbb{P}$ and so $s \nmid t$. It follows that, by Definition~\ref{def:GFG}, there does not exist $(s,t) \in A_n$.
\end{proof}

\begin{gfgp}
A vertex $v \in V_n$ induces a disconnected component in a graph $F_n$ if and only if $v^e=n-v$ for some $e \in \mathbb{N}_1$.
\label{gfgp:DiscComp}
\end{gfgp}
\begin{proof}
Suppose that the condition $v^e=n-v$ is satisfied for $v \in V_n$. By Definition~\ref{def:GFG}, $v$ has a loop and has no other predecessors, because $q \mid (n-v)$ if and only if $q=v$. According to GFG Property~\ref{gfgp:Loop2}, $v$ has no successors, except for $v$ itself. It follows that $v$ induces a disconnected component. Conversely, assume that $v \in V_n$ induces a disconnected component in $F_n$. The vertex $v$ has no predecessors (except for the loop), thus, by Definition~\ref{def:GFG}, if $v \neq u \in V_n$, then $u \nmid (n-v)$. Therefore, $v^e=n-v$ for some $e \in \mathbb{N}_1$.
\end{proof}

\begin{gfgp}
If two vertices $v_1,v_2 \in V_n$ in a graph $F_n$ represent a Goldbach partition, i.e. $v_1+v_2=n$, then these vertices induce a source strongly connected component in $F_n$.
\label{gfgp:GFGpartition}
\end{gfgp}
\begin{proof}
The condition $v_1+v_2=n$ and Definition~\ref{def:GFG} imply that there exist arcs $(v_1,v_2)$ and $(v_2,v_1)$, and there are no other input arcs to $v_1$ and $v_2$ in $F_n$. Therefore, the vertices $v_1$ and $v_2$ belong to the cycle $(v_1,v_2)$ and so, they belong to the same SCC. Moreover, this SCC does not include more vertices and it has no predecessors, because it does not exist any path from $v\in V_n\setminus \{v_1,v_2\}$ to $v_1$ or $v_2$.
\end{proof}

\begin{gfgp}
All Goldbach factorization graphs, except for $F_4$ and $F_6$, contain vertices without loops.
\label{gfgp:AllGFGloops}
\end{gfgp}
\begin{proof}
According to properties of the Ramanujan primes, $\pi(x)-\pi(x/2)\geq 2$ for $x\geq 11$%
~\cite{Ramanujan1919}. Hence, $\#\left([k+1,2k]\cap\mathbb{P}\right)\geq 2$ for $k \in \mathbb{N}_0$, $2k\geq 12$ and, consequently, if $n$ is even, $n\geq 12$, and $H_n=[n/2+1,n-2]\cap\mathbb{P}$ then $\#H_n\geq 1$. It is obvious that if $p\in H_n$, then $p \in V_n$ and $p\nmid n$. Thus, by GFG Property~\ref{gfgp:Loop1}, a graph $F_n$ has vertices without loops for $n\geq 12$. A direct inspection can be performed for $n<12$. In particular, $V_4=\{2\}$ and $V_6=\{2,3\}$, hence, all the vertices of $F_4$ and $F_6$ have loops. In the remaining cases, vertices without loops exist, e.g. $3$ is such a vertex in $F_8$ and $F_{10}$, because $3\in V_8$ and $3\nmid 8$, as well as, $3\in V_{10}$ and $3\nmid 10$.
\end{proof}

We will classify autonomous components of a Goldbach factorization graph $F_n$ into three types:
\begin{enumerate}
\item A \textit{Goldbach autonomous component} (GAC) is an autonomous component induced by vertices $v_1,v_2 \in V_n$ representing a Goldbach partition, i.e. such that $v_1+v_2=n$. A GAC is induced by two vertices of $F_n$ if $v_1 \neq v_2$, and it is induced by a single vertex otherwise.
\item A \textit{trivial autonomous component} (TAC) is an autonomous component induced by a single vertex which is not a GAC.
\item An \textit{exceptional autonomous component} (EAC) is any autonomous component which is neither a TAC nor a GAC.
\end{enumerate}
\begin{remark}
TACs are structurally very simple, as they consist of a single vertex, but they are considered as trivial in this work because of their redundancy. TACs have no successors in GFGs, except for being successors of themselves, according to GFG Properties~(\ref{gfgp:Loop1})-(\ref{gfgp:Loop2}), therefore, they cannot provide Goldbach factors for other primes, in contrast to GACs (induced by two vertices) and EACs. One can omit all TACs in a GFG, and its essential properties will not change, in particular, the properties important in the context of Theorem~\ref{th:EACGoldbach}. TACs are included into the collection of autonomous components to obtain a complete and consistent form of GFGs, based on all the vertices from $\mathbb{V}_n$. In the practical use of GFGs, e.g. an algorithmic processing, TACs can be dropped.
\end{remark}
We have the following simple observations concerning TACs and GACs.
\begin{proposition}
There exist infinitely many numbers $n \in \mathbb{N}_\mathrm{G}$ for which $F_n$ contains a trivial autonomous component.
\end{proposition}
\begin{proof}
Choose arbitrary $v \in \mathbb{P}$. According to GFG Property~\ref{gfgp:DiscComp}, if $v^e=n-v$, then $v$ induces a disconnected component in $F_n$. This component is a TAC if $e \in \mathbb{N}_2 = \mathbb{N}_1 \setminus \{1\}$. Therefore, $F_n$ has a TAC induced by $v$ if $n \in \mathbb{N}_\textrm{T}^{(v)}=\left\{v+v^e : e \in \mathbb{N}_2\right\}$, and $\mathbb{N}_\textrm{T}^{(v)}$ is infinite.
\end{proof}
\begin{proposition}
There exist infinitely many numbers $n \in \mathbb{N}_\mathrm{G}$ for which $F_n$ contains a Goldbach autonomous component.
\end{proposition}
\begin{proof}
According to the classic result of number theory, the set of primes is infinite. Hence, the equation
\begin{equation}
	p+q=n, \qquad p,q \in \mathbb{P}
\label{eq:PropGAC}
\end{equation}
has solutions for infinitely many $n \in \mathbb{N}_\mathrm{G}$. If $n$ satisfies~(\ref{eq:PropGAC}), $F_n$ contains a GAC, according to GFG Property~\ref{gfgp:GFGpartition} and the classification of autonomous components.
\end{proof}

While infinitude of TACs and GACs is obvious, it is not the case for EACs. This is not surprising, because the knowledge about the cardinality of the EACs' set would represent an important information related to the BGC. The respective correlation is expressed by the following theorem, which brings out how studying of EACs relates to the BGC.  
\begin{theorem}
If $n \in \mathbb{N}_\mathrm{G}$ does not satisfy the binary Goldbach conjecture, then the graph $F_n$ contains an exceptional autonomous component.
\label{th:EACGoldbach}
\end{theorem}
\begin{proof}
It is obvious that the BGC is satisfied for 4 and 6. On account of GFG Property~\ref{gfgp:AllGFGloops}, each graph $F_n$ for $n>6$ has vertices without loops. Each vertex in $F_n$ has a predecessor, according to Definition~\ref{def:GFG}, and a vertex with a loop cannot precede a vertex without a loop, by GFG~Property~\ref{gfgp:Loop2}, therefore, the vertices without loops form at least one source SCC with cycles based on more than one vertex. It follows that, according to GFG Property~\ref{gfgp:AutCmp}, each graph $F_n$, where $n>6$, has at least one autonomous component induced by at lest two vertices and, if $n$ does not satisfy the BGC, none of these components is a GAC, thus, they are EACs.
\end{proof}

\section{Twin EACs}
\label{sec:TwinEACs}
Let us distinct a specific variant of EACs.
\begin{definition}
A \textit{twin exceptional autonomous component} (twin EAC) is an EAC induced by exactly two vertices of a related GFG.
\end{definition}

We will study the existence of twin EACs separately, because, in this specific case, it can be done using the single diophantine equation
\begin{equation}
	a^x+b=b^y+a,
\label{eq:TwinEAC}	
\end{equation}
according to the following theorem.

\begin{theorem}
A set of two primes $\{a,b\}$ represents vertices inducing a twin EAC in some graph $F_n$ if and only if $a,b>2$ and the diophantine equation $a^x+b=b^y+a$ is satisfied, where $x,y>1$, $x \neq y$.
\label{th:TwinEAC}
\end{theorem}
\begin{proof}
Consider the set of primes $\{a,b\}$. If $a=2$ or $b=2$, the set does not induce a multi-vertex autonomous component, because, by GFG Properties~\ref{gfgp:Loop1}~and~\ref{gfgp:Loop2}, the vertex $v=2$ has no successors different than 2, so it does not belong to any cycle, except for a loop. If $a,b>2$ and the equation $a^x+b=b^y+a$ is satisfied for $x,y \in \mathbb{N}_1$, we obtain $a^x+b=n$ (*) and $b^y+a=n$ (**), where $n$ is even and $a,b \in [2,n-2] \cap \mathbb{P}$. According to Definition~\ref{def:GFG}, it follows from (*) and (**) that $a$ and $b$ are vertices of $F_n$ such that $a$ has an input arc only from $b$ and vice-versa, therefore, the set $\{a,b\}$ induces a two-vertex autonomous component. This component is a GAC if $x=1$ and $y=1$, and it is a twin EAC in other cases. However, the equation $a^x+b=b^y+a$ cannot be satisfied if $x+y>2$ and ($x=1$ or $y=1$), because, under the given conditions, this equation obtains the form $p^e+q=q+p$, where $p,q,e>1$, which is obviously a contradiction. It is also impossible that $x,y>1$ and $x=y$, indeed, we have $a^e-b^e=(a-b)\sum_{i=1}^e(a^{e-i}b^{i-1}) \neq a-b$. It follows that, if $a,b>2$ and the diophantine equation $a^x+b=b^y+a$ is satisfied for $x,y > 1$, $x \neq y$, then the set $\{a,b\}$ induces a twin EAC.

In turn, assume that a graph $F_n$ has a twin EAC induced by a set of vertices $\{a,b\}$. We obtain $a,b>2$, because $a,b \in \mathbb{P}$ and, as stated above, the vertex $v=2$ cannot belong to a two-vertex autonomous component. According to Definition~\ref{def:AC}, the vertices $a$ and $b$ have input arcs only from vertices from the set $\{a,b\}$. On account of GFG Property~\ref{gfgp:AutCmp}, the set $\{a,b\}$ induces a SCC, so there exist paths, in the form of single arcs, from $a$ to $b$ and from $b$ to $a$. According to GFG Property~\ref{gfgp:Loop2}, if $(s,t) \in A_n$ for $s \neq t$, then $(s,s) \not\in A_n$. Therefore, the vertex $a$ has only one input arc from the vertex $b$ and vice-versa. It implies that the diophantine equations $a^x+b=n$ and $b^y+a=n$ are satisfied for $x,y \in \mathbb{N}_1$, according to Definition~\ref{def:GFG}, and consequently, we have $a^x+b=b^y+a$. The restrictions $x,y>1$, $x \neq y$, derived above, are satisfied in the considered case as well, because they represent the conditions that the set $\{a,b\}$ does not induce a GAC and that the equation $a^x+b=b^y+a$ can have a solution. Finally, if a set of primes $\{a,b\}$ induces a twin EAC, then $a,b>2$ and the diophantine equation $a^x+b=b^y+a$ is satisfied for $x,y>1$, $x \neq y$.
\end{proof}

The equation~(\ref{eq:TwinEAC}) is a special form of the well known Pillai diophantine equation%
~\cite{Bennett2001} which has been intensively studied in the basic variant
\begin{equation}
	a^x-b^y=c
\label{eq:PillaiBasic}
\end{equation}
and in many other modified forms. Pillai proved that for any fixed pair $(a,b)$ of coprime positive integers there exists $c_0(a,b)$ such that, if $c > c_0(a,b)$, then the diophantine equation~(\ref{eq:PillaiBasic}) has at most one solution $(x,y)$%
~\cite{Pillai1936,Pillai1937}.

Many results related to the Pillai equation lead to corollaries concerning twin EACs. Some of them are given below.
\begin{corollary}
If a set of primes $\{a,b\}$, $a > b$, induces a twin exceptional autonomous component, then $b/a<1697/1698$.
\label{cor:BAEACs}
\end{corollary}
\begin{proof}
Terai proved that the diophantine equation~(\ref{eq:PillaiBasic}) under the conditions $a-b=c>1$, $(a,b)=1$, and $b \geq 1697c$ has no positive integral solutions $(x,y)$, except for the trivial one $(x,y)=(1,1)$%
~\cite{Terai1999}~(Th.~3). If a set $\{a,b\}$ induces an EAC, then $(a,b)=1$, because $a,b \in \mathbb{P}$, moreover $a-b=c>1$ and $(x,y) \neq (1,1)$, according to Theorem~\ref{th:TwinEAC}. Therefore, if the set $\{a,b\}$ induces an EAC, then $b<1697c$, and consequently $b/a<1697/1698$.
\end{proof}

\begin{corollary}
There exist at most finitely many twin exceptional autonomous components induced by a family of sets of primes $\{a,b\}$, where a is fixed.
\end{corollary}
\begin{proof}
Luca proved that the diophantine equation
\begin{equation}
	p^{x_1}-p^{x_2}=q^{y_1}-q^{y_2},\qquad \textrm{where }x_1 \neq x_2, \; p \neq q, \; p,q  \in \mathbb{P}
\label{eq:Luca}
\end{equation}
has at most finitely many solutions $\left(p,x_1,y_1,x_2,y_2\right)$ for a fixed $q$%
~\cite{Luca2003}~(Th.~1). Hence, we obtain immediately, that the equation~(\ref{eq:TwinEAC}) has at most finitely many solutions for a fixed $a$ and under the conditions $a,b \in \mathbb{P}$, $a,b>2$, $a \neq b$, $x,y>1$, $x \neq y$. Therefore, according to Theorem~\ref{th:TwinEAC}, at most finitely many pairs $(a,b) \in \mathbb{P}^2$ with a fixed $a$ induce twin EACs.
\end{proof}

\begin{corollary}
Assuming the abc conjecture, there exist only finitely many twin exceptional autonomous components.
\label{cor:AbcEACs}
\end{corollary}
\begin{proof}
It has also been proved by Luca, that the equation~(\ref{eq:Luca}) has at most finitely many solutions $\left(p,q,x_1,y_1,x_2,y_2\right)$ under the abc conjecture%
~\cite{Luca2003}~(Th.~2). This directly implies the considered corollary.
\end{proof}
\begin{remark}
Corollary~\ref{cor:AbcEACs} is also implied by the result obtained by Mignotte and Pethő, related to the diophantine equation $x^p-x=y^q-y$%
~\cite{Mignotte1999}~(Th.~2).
\end{remark}

Corollaries~(\ref{cor:BAEACs})-(\ref{cor:AbcEACs}) indicate that numerous proved properties of the~diophantine equation~(\ref{eq:PillaiBasic}) are useful to analyse twin EACs. However, there exists a property which allows to formulate the following conclusive theorem.
\begin{theorem}
There exists only one twin exceptional autonomous component. It is induced by the set of vertices $\{3,13\}$ and it is included in $F_{2200}$.
\label{th:OneTwinEAC}
\end{theorem}
\begin{proof}
This theorem is in fact a corollary of Theorem~3 proved by Scott in%
~\cite{Scott1993}, stating that the diophantine equation~(\ref{eq:PillaiBasic}) under the restrictions $a \in \mathbb{P}$, $b>1$, $c>0$ has at most one solution $(x,y) \in {\mathbb{N}_1}^2$, except for five explicitly enumerated counterexamples for which exactly two solutions exist. Assume that a set of primes $\{a,b\}$, $a>b$, induces a twin EAC. On account of Theorem~\ref{th:TwinEAC}, it implies that the equation $a^x+b=b^y+a$ is satisfied, where $x,y>1$, $a,b \in \mathbb{P}$, and $a,b>2$. It further implies that the equation $a^x-b^y=a-b=c$ under the constraints formulated in the Scott theorem is satisfied, as well, because these constraints are weaker than the constraints imposed on the values of $a$, $b$, $x$, and $y$ in Theorem~\ref{th:TwinEAC}. Additionally, the equation $a^x-b^y=a-b=c$ is trivially satisfied for $(x,y)=(1,1)$. Thus, we obtain two solutions $(x,y)$ for a fixed triple $(a,b,c)$, the basic solution with $x,y>1$ and the trivial one. According to the Scott theorem, there exist only five triples $(a,b,c)$ for which the considered equation has more than one solution $(x,y) \in {\mathbb{N}_1}^2$. By direct inspection of these solutions, we find that $(a,b,c,x,y)=(13,3,10,3,7)$ is the only assignment with $a$ and $b$ such that $\{a,b\}$ induces a twin EAC. From the equation $a^x+b=b^y+a=n$, we obtain that the found twin EAC is included in $F_{2200}$.
\end{proof}

\section{Computer-aided search for EACs}
\label{sec:EACSearch}
According to GFG Property~\ref{gfgp:AutCmp}, to check the existence of an EAC in a GFG, one can enumerate SCCs of the GFG and apply simple conditions for each of the SCCs, which extract EACs from among other autonomous components. However, we want to check the EAC existence for as many GFGs as possible, without extracting individual EACs. The additional extraction may be performed only for the GFGs including EACs. For that reason, we have prepared a simplified and efficient algorithm for checking EACs existence, which does not enumerate SCCs explicitly. Properties of autonomous components of GFGs lead to the following theorem which points the used algorithmic way to search for the GFGs having EACs.
\begin{theorem}
Let $F_n=\left(V_n,A_n,w_n\right)$ be a Goldbach factorization graph. Let $\Gamma_n \subseteq V_n$ be the set of all the vertices which induce Goldbach autonomous components, i.e. $\Gamma_n=\left\{v \in V_n : \exists_{u \in V_n}(v+u=n)\right\}$. Let $\Xi_n \subset V_n$ be the set of all the vertices reachable from some vertex in $\Gamma_n$. Let $\Theta_n \subseteq V_n$ be the set of all the vertices which induce trivial autonomous components, i.e. $\Theta_n=\left\{v \in V_n : \exists_{e \in \mathbb{N}_1 \setminus \{1\}}(v^e=n-v)\right\}$. The graph $F_n$ has an exceptional autonomous component if and only if $(V_n \setminus \Theta_n) \setminus (\Gamma_n \cup \Xi_n) \neq \emptyset$.
\label{th:EACCheck}
\end{theorem}
\begin{proof}
Let $\Omega_n=\left\{\Gamma_n, \Theta_n, \Upsilon_n\right\}$, where $\Upsilon_n \subseteq V_n$ is the set of all the vertices which induce EACs. If $X,Y \in \Omega_n$, $X \neq Y$, $v_1 \in X$, $v_2 \in Y$, then $v_1 \neq v_2$ and $v_2$ is not reachable from $v_1$, because $v_1$ and $v_2$ belong to separate source SCCs. We obtain that, if $v \in \Upsilon_n$, then $v \notin \Theta_n$ and $v \notin \Gamma_n$ and $v \notin \Xi_n$. Therefore, if an EAC exists, i.e. $\Upsilon_n \neq \emptyset$, then $\Upsilon_n \subseteq (V_n \setminus \Theta_n) \setminus (\Gamma_n \cup \Xi_n) \neq \emptyset$. In turn, if $\Upsilon_n = \emptyset$, then for each $v \in V_n$, such that $v \notin \Gamma_n$ and $v \notin \Theta_n$, we have $v \in \Xi_n$ and, for that reason, $(V_n \setminus \Theta_n) \setminus (\Gamma_n \cup \Xi_n) = \emptyset$.
\end{proof}

Theorem~\ref{th:EACCheck} provides a very high level pseudocode of the program for detection of EACs. The program has to construct the sets $V_n$, $\Theta_n$, $\Gamma_n$, $\Xi_n$, derive the set $(V_n \setminus \Theta_n) \setminus (\Gamma_n \cup \Xi_n)$, and check if it is empty. A simple but complete executable version of such a program is given in Listing~\ref{lstProgr1} in the form of a Scilab script. For a given $n \in \mathbb{N}_\mathrm{G}$ (line 2) and the related graph $F_n$, the script constructs the set of vertices $V_n \setminus \Theta_n$ (lines 4--14). In particular, the set $V_n$ is generated in the line 4 and the vertices from $\Theta_n$ are dropped with the use of the condition from the line 9. For internal representation of the graph model, the vertices are mapped to consecutive natural numbers contained in the vector \texttt{vertices} (line 14), whereas the mapping itself is defined by the vectors \texttt{vertexToPrime} (lines 5, 12) and \texttt{primeToVertex} (line 20). This mapping makes it possible to use the vertex value as an index of a successor matrix in the remainder of the script. In the lines 17--24, a set of vertices which induce GACs is generated and inserted into the vector \texttt{gacVertices}. The arcs adjacent to the vertices $V_n \setminus \Theta_n$ are determined in the lines 26--36. After the line 36, a complete model of the subgraph of $F_n$ induced by the vertices $V_n \setminus \Theta_n$ is available, such that the $i$-th vertex represents the prime number \texttt{vertexToPrime(i)}, the number $\sigma$ of successors of this vertex is equal to \texttt{gfgSuccNum(i)}, and the successors are defined by the values \texttt{gfgSuccessors(i, c)}, where $\texttt{c} \in \{1, 2, \ldots, \sigma\}$. The set of vertices inducing GACs is extended by the vertices reachable from this set (lines 38--49) and the resultative set $\Gamma_n \cup \Xi_n$ is represented by the vector \texttt{gacAndSuccessors}. The checking set $(V_n \setminus \Theta_n) \setminus (\Gamma_n \cup \Xi_n)$ is derived (line 51) and, finally, its cardinality is printed as the result (line 52). According to Theorem~\ref{th:EACCheck}, the verified graph $F_n$ has an EAC if and only if the printed result is greater than zero.

\begin{lstlisting}[xleftmargin=20pt,numbers=left,basicstyle=\small,caption={The Scilab script for checking the existence of EACs},label=lstProgr1]

n = 1928; // set the number to check

pms = primes(n - 2);
vertexToPrime = [];
for p = pms
    fcs = factor(n - p);
    uFcs = unique(fcs);
    if length(fcs) > 1 & length(uFcs) == 1 & uFcs == p
        continue;
    end
    vertexToPrime = [vertexToPrime p];
end
vertices = 1:length(vertexToPrime);
vNum = length(vertices);

gacVertices = [];
for v = vertices
    p = vertexToPrime(v);
    primeToVertex(p) = v;
    if length(factor(n - p)) == 1
        gacVertices = [gacVertices v];
    end
end

gfgSuccessors = zeros(vNum, vNum);
gfgSuccNum = zeros(1, vNum);
for v = vertices
    p = vertexToPrime(v);
    primePredecessors = unique(factor(n - p));
    for p = primePredecessors
        vPred = primeToVertex(p);
        gfgSuccNum(vPred) = gfgSuccNum(vPred) + 1;
        gfgSuccessors(vPred, gfgSuccNum(vPred)) = v;
    end
end

gacAndSuccessors = gacVertices;
notVisited = gacVertices;
while ~isempty(notVisited)
    notVisitedSave = notVisited;
    for v = notVisitedSave
        notVisited = union(notVisited, ..
        	gfgSuccessors(v, 1:gfgSuccNum(v)));
    end
    notVisited = setdiff(notVisited, gacAndSuccessors);
    gacAndSuccessors = union(..
    	gacAndSuccessors, notVisited);
end

checkSet = setdiff(vertices, gacAndSuccessors);
mprintf('%d -> %d\n', n, length(checkSet));
\end{lstlisting}

The script from Listing~\ref{lstProgr1} has been prepared to demonstrate the developed algorithm in a simple and concise way. As a result, the implementation is not optimized respective to execution speed. However, it has been experimentally verified that this script, run on a modern personal computer, checks the existence of EACs in all the graphs $F_n$ for $n \leq 10^4$ in tens of minutes.

To increase the number of the checked GFGs, a more refined program has been prepared. It uses, among others, pre-factorization of all the odd integers needed for GFGs construction, permanent worst-case memory allocation for the generated graph objects, and multi-thread processing. After a few months of experiments (the time equivalent for a single modern personal computer), the existence of EACs in the graphs $F_n$ has been checked for $n \leq 10^8$. Only 6 graphs $F_n$ containing EACs have been found, namely, for $n \in \{128,1718,1862,1928,2200,6142\}$ and, as expected, only the graph $F_{2200}$ contains a twin EAC.

\section{Review of found EACs}
\label{sec:EACReview}
The number of the found EACs is small enough to present them individually. The EACs have no more than 28 vertices and 64 arcs, therefore, we expect that many of their properties can be determined precisely, even if the related algorithms are intractable. Moreover, graphs of such size can be graphically presented in a legible form. The GFGs containing EACs are relatively small, as well, and the largest one, $F_{6142}$, has 800 vertices and 1732 arcs. We provide a review of some basic properties of the found EACs and the GFGs which contain them.

The graph $F_{128}$ is presented in Figure~\ref{fig:GFG128}. It is drawn almost explicitly, as it has only 30 vertices, however, the arcs from the vertices in the cells SOURCE--EXCEPTIONAL and INNER--EXCEPTIONAL to the vertices in the cell SINK--EXCEPTIONAL are hidden for legibility. The graph contains three GACs, and one EAC. The EAC is induced by the vertices 3, 5, 7, 11, 13, 23, 29, and 41.
\begin{figure}
	\centering
	\includegraphics[width=\textwidth,trim={53mm 126mm 42mm 43mm},clip]{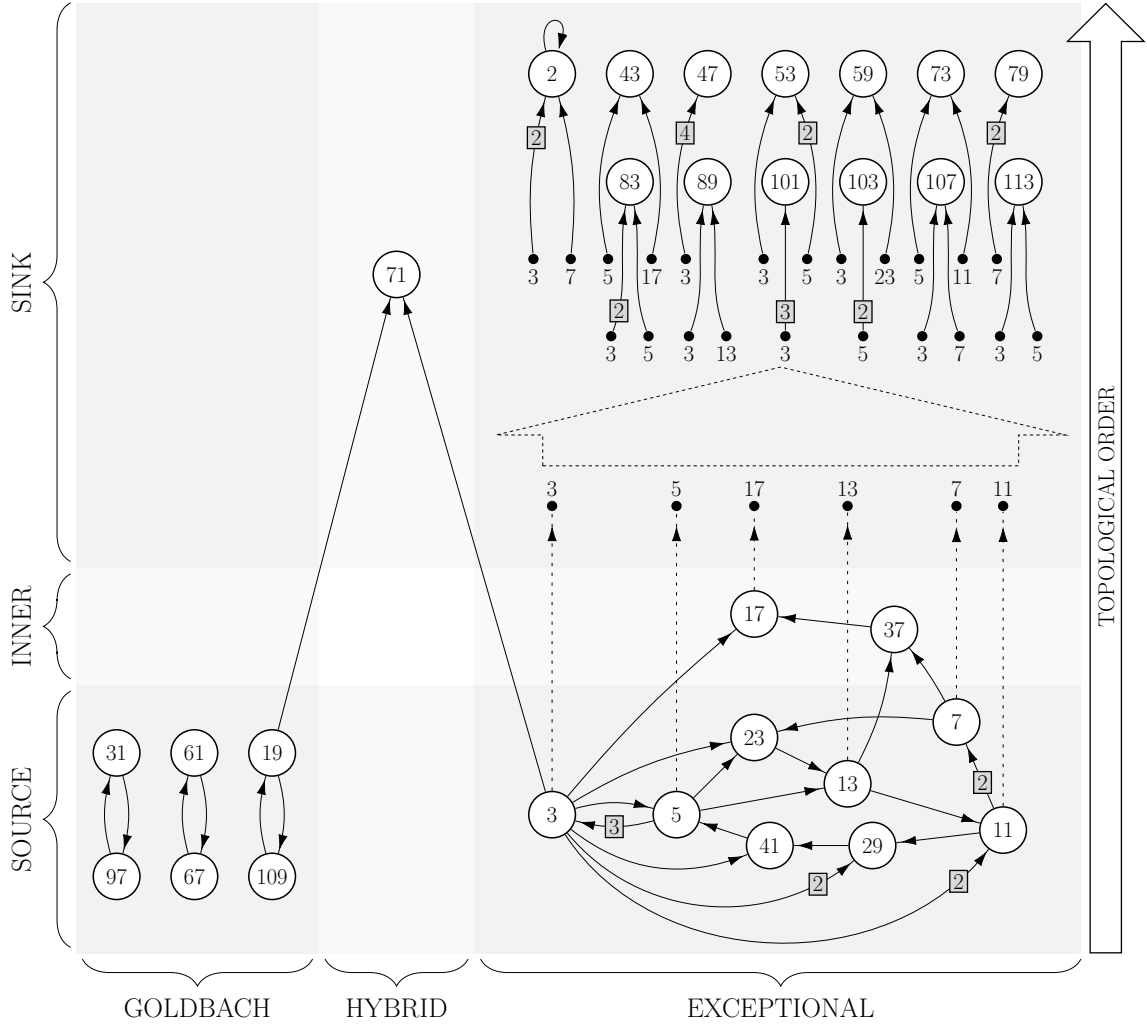}
	\caption{The Goldbach factorization graph $F_{128}$}
	\label{fig:GFG128}
\end{figure}
In the GFG representation from Figure~\ref{fig:GFG128}, a concept of \textit{condensation map} is used, which is defined as follows.
\begin{definition}
Let us define two sets of labels
\begin{align*}
	&\mathbb{L}_\mathrm{R}=\left\{\text{SOURCE},\text{INNER},\text{SINK}\right\},\\
	&\mathbb{L}_\mathrm{C}=\left\{\text{GOLDBACH},\text{HYBRID},\text{EXCEPTIONAL}\right\}.
\end{align*}
A \textit{condensation map} of a Goldbach factorization graph $F_n=\left(V_n,A_n,w_n\right)$ is the function $\mu_n : V_n \mapsto \mathbb{L}_\mathrm{R} \times \mathbb{L}_\mathrm{C}$, such that $(x,y)=\mu_n(v)$, where
\[
	x=\begin{cases}
		\text{SOURCE} &\text{if $v$ belongs to a source or disconected SCC of $F_n$},\\
		\text{SINK} &\text{if $v$ belongs to a sink SCC of $F_n$},\\
		\text{INNER} &\text{otherwice},
	\end{cases}
\]
\[
	y=\begin{cases}
		\text{GOLDBACH} &\parbox{20em}{if $v$ does not induce a TAC\\and it is not reachable from any EAC,}\\[1.1em]
		\text{EXCETIONAL} &\parbox{20em}{if $v$ does not induce a TAC\\and it is not reachable from any GAC,}\\[0.7em]
		\text{HYBRID} &\text{otherwice}.
	\end{cases}
\]
\label{def:CondMap}
\end{definition}

The condensation map is represented graphically as a matrix of 9 cells delimited by the intersection of 3 horizontal and 3 vertical strips. The cells map one-to-one to the codomain elements of the function $\mu_n$. The condensation map is useful to roughly group the vertices of a GFG in the way which emphasizes the concept of GACs and EACs, as well as their relationships with other vertices in a GFG.

We will use the condensation map to present some information about the condensation structure of the GFGs containing EACs. The number of vertices of these graphs is too large to draw all them explicitly and legibly, therefore, we provide only the distribution of SCCs over the cells defined by the condensation map (Figure~\ref{fig:GFGmaps}). Each annotation of the form \texttt{$\alpha$ x $\beta$} inside a cell denotes that the cell contains $\alpha$ SCCs and each one is induced by $\beta$ vertices.

Let us notice that the condensation map is not completely symmetric respective to types of SCCs. Disconnected SCCs are simultaneously sources and sinks. They have been, however, assigned arbitrarily to the SOURCE strip in Definition~\ref{def:CondMap}. It makes it possible to group conveniently connected and disconnected autonomous components in this strip, as the connected autonomous components are always sources. In particular, disconnected and connected GACs are grouped together in the cell SOURCE-GOLDBACH, whereas TACs are assigned to the cell SOURCE-HYBRID. However, none of the found GFGs with EACs contains TACs, hence, all the cells SOURCE-HYBRID are empty in Figure~\ref{fig:GFGmaps}.
\begin{figure}
	\centering
	\includegraphics[width=\textwidth,trim={27mm 148mm 24mm 20mm},clip]{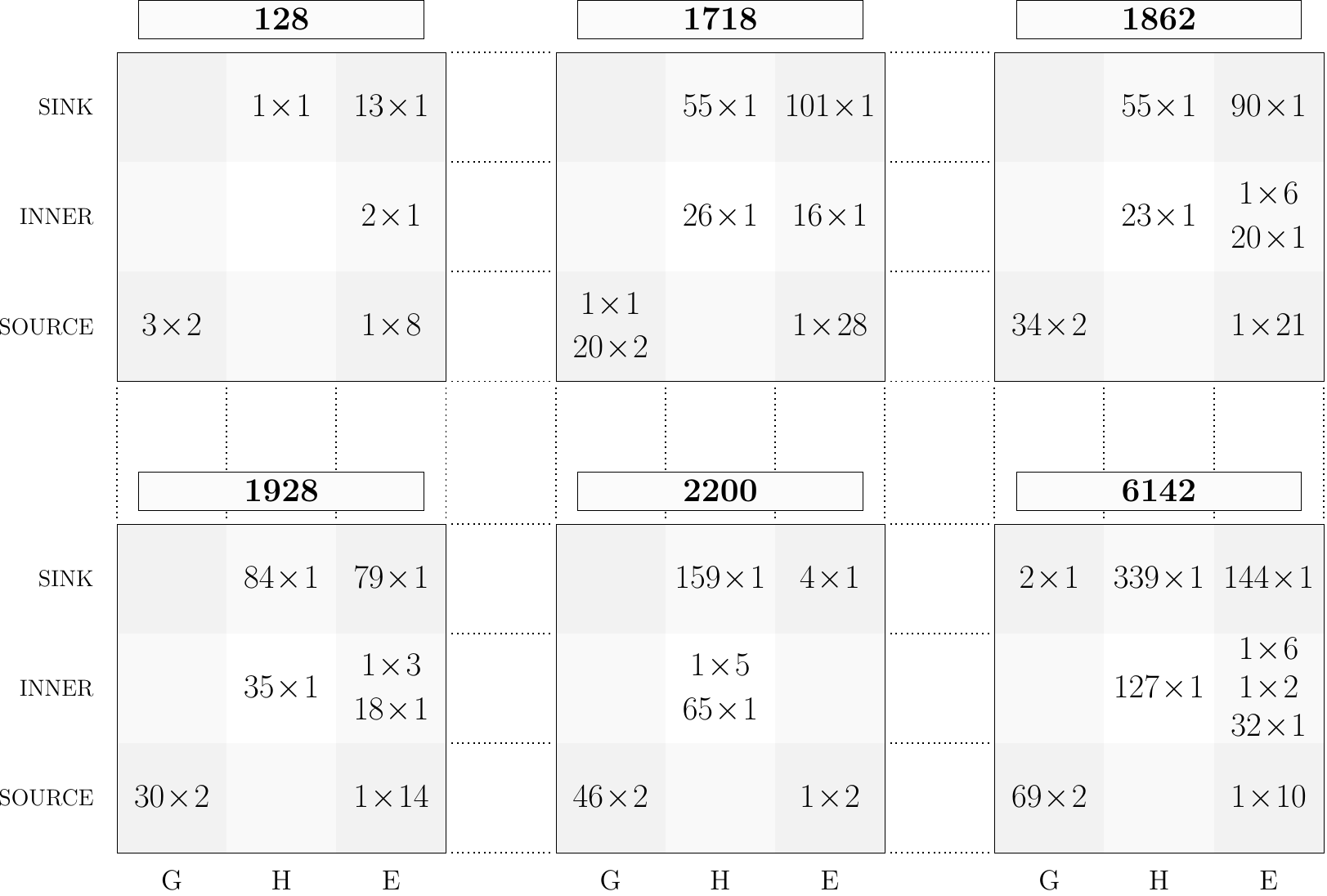}
	\caption{The SCCs distribution over the condensation maps of~the~found GFGs containing EACs}
	\label{fig:GFGmaps}
\end{figure}

In Figures~\ref{fig:EAC1718}--\ref{fig:EAC2200and6142}, all the found EACs are shown, except for the EAC contained in $F_{128}$, which is explicitly rolled out in Figure~\ref{fig:GFG128}.
\begin{figure}
	\centering
    \includegraphics[width=0.8\textwidth,trim={54mm 116mm 42mm 44mm},clip]{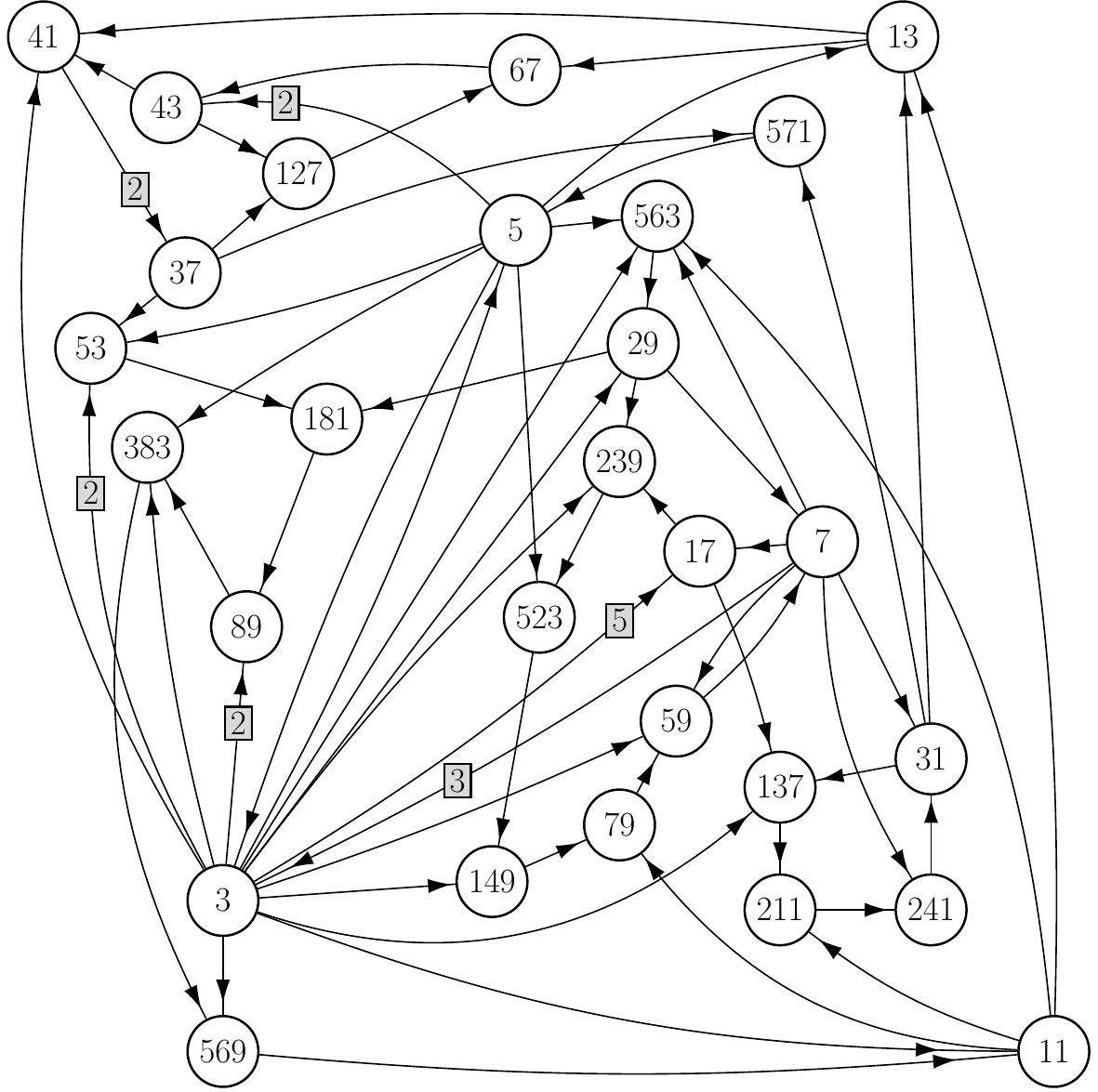}
	\caption{The EAC contained in $F_{1718}$}
	\label{fig:EAC1718}
\end{figure}
\begin{figure}
	\centering
    \includegraphics[width=0.8\textwidth,trim={53mm 138mm 42mm 44mm},clip]{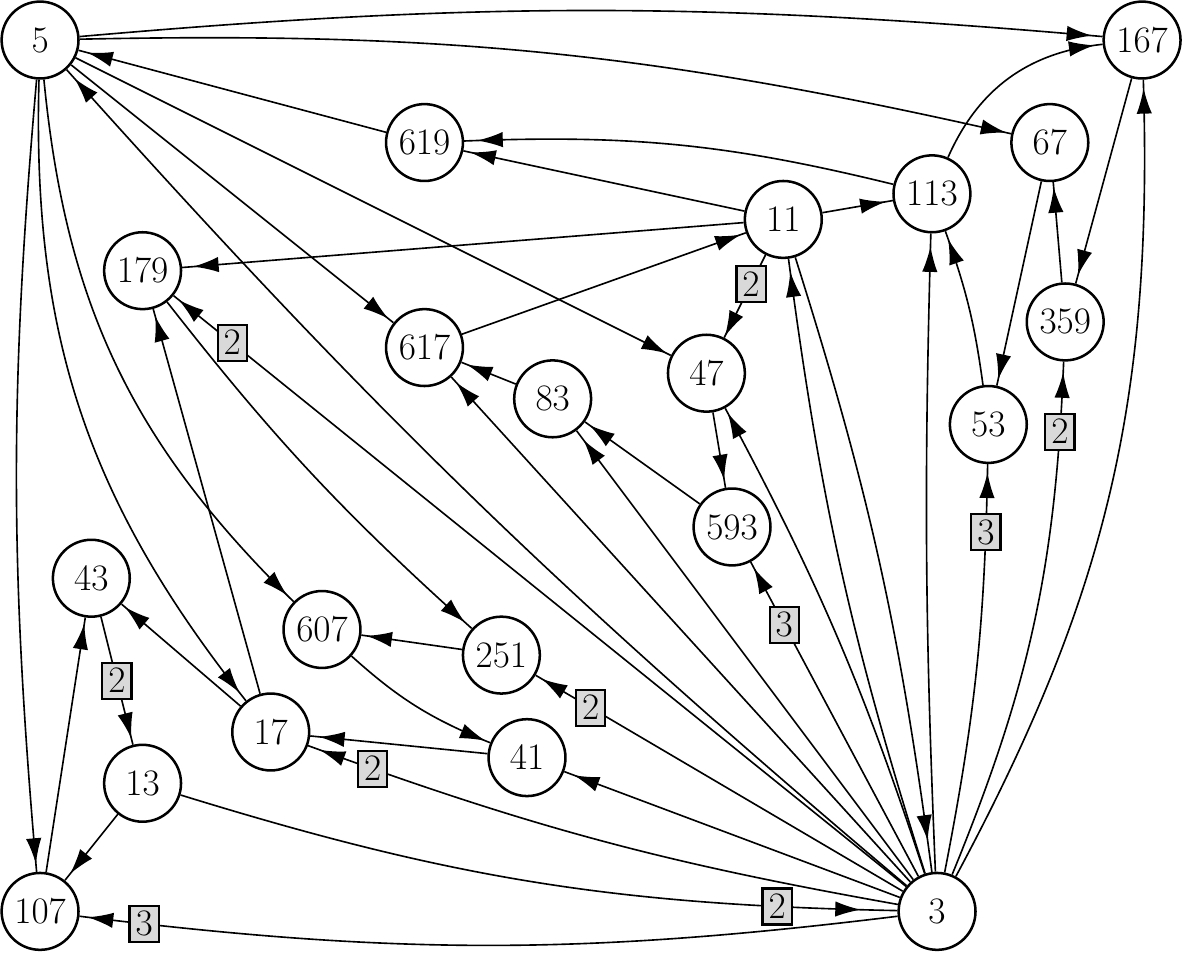}
	\caption{The EAC contained in $F_{1862}$}
	\label{fig:EAC1862}
\end{figure}
\begin{figure}
	\centering
    \includegraphics[width=0.7\textwidth,trim={53mm 134mm 43mm 44mm},clip]{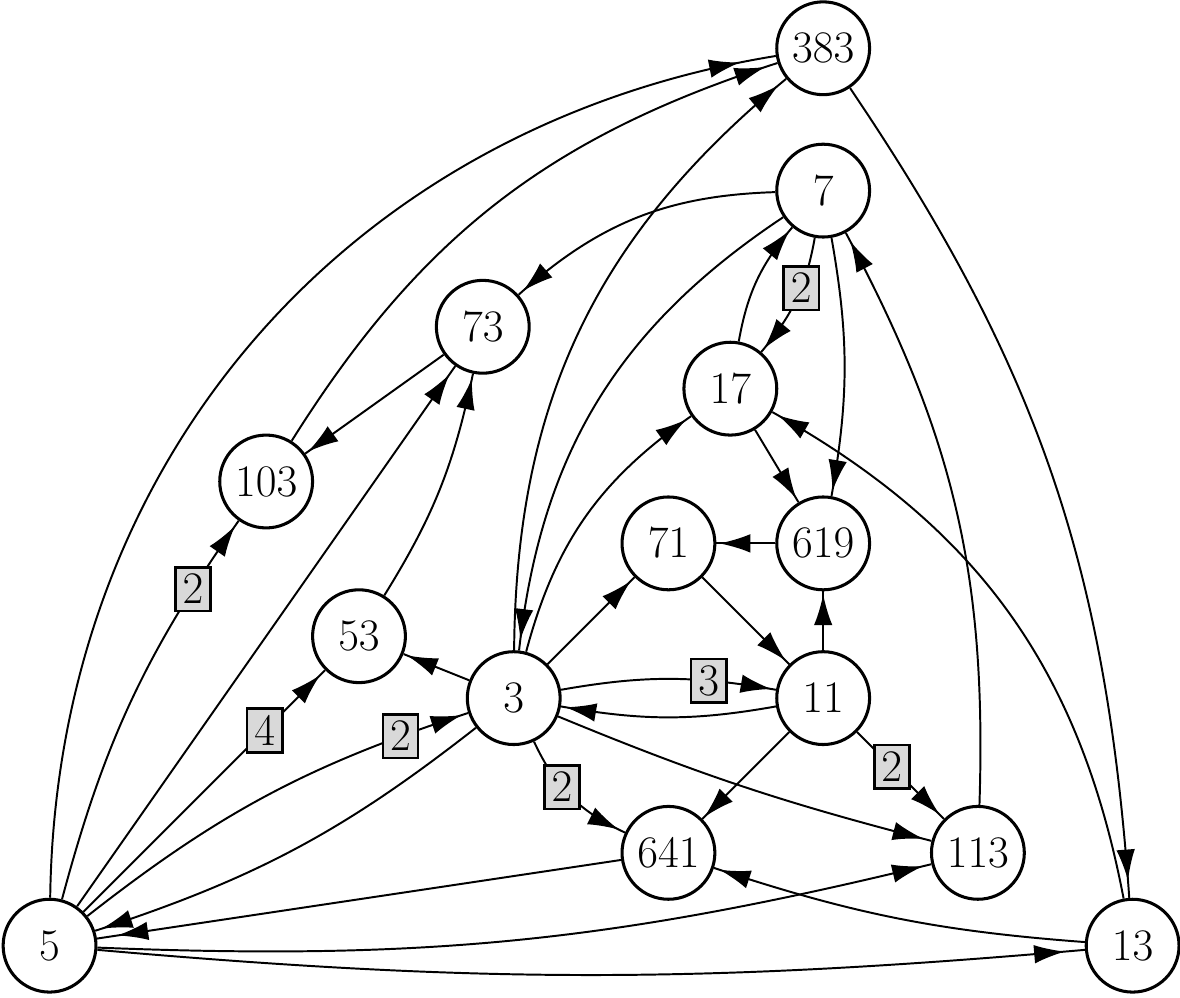}
	\caption{The EAC contained in $F_{1928}$}
	\label{fig:EAC1928}
\end{figure}
\begin{figure}
	\centering
	\begin{subfigure}[b]{0.25\textwidth}
		\centering
    	\includegraphics[width=\textwidth,trim={54mm 198mm 44mm 44mm},clip]{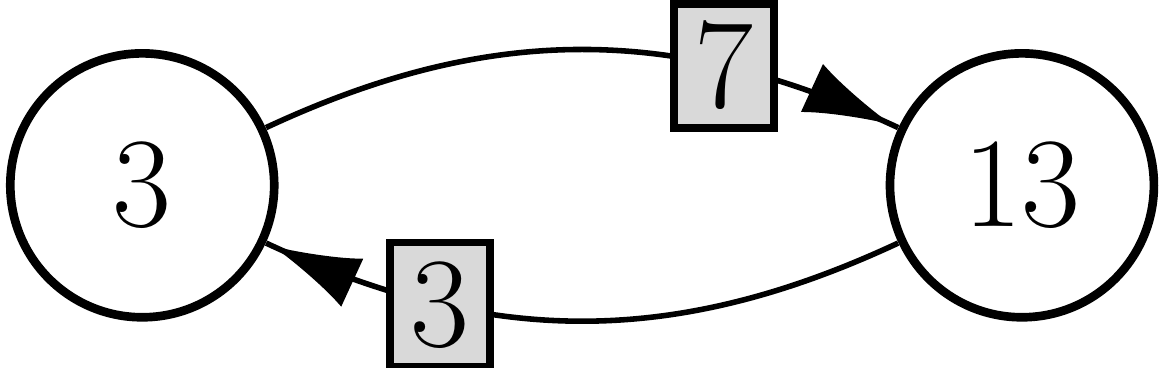}
		\caption{$n=2200$}
	\end{subfigure}
	\quad\quad
	\begin{subfigure}[b]{0.5\textwidth}
		\centering
    	\includegraphics[width=\textwidth,trim={55mm 140mm 43mm 44mm},clip]{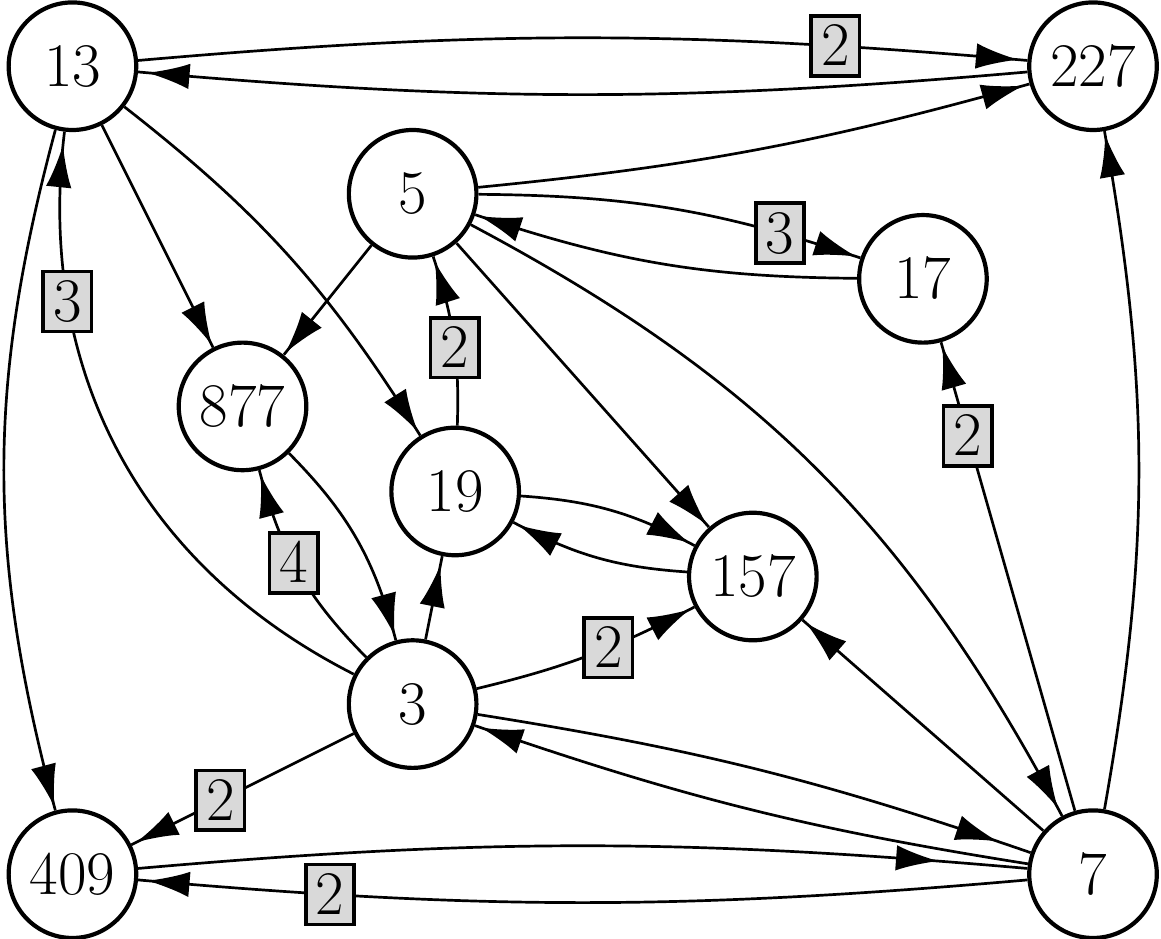}
    	\caption{$n=6142$}
	\end{subfigure}
	\caption{The EACs contained in  $F_{2200}$ and $F_{6142}$}
	\label{fig:EAC2200and6142}
\end{figure}

In the remainder of this section, first, we present selected numerical parameters describing the found EACs or GFGs which contain them, and then, we explain how these parameters have been obtained, to justify their reliability and to make it possible to replicate the related computations independently.

\subsection{Numerical parameters of the found EACs}
The basic numerical parameters are given in Table~\ref{tab:GFGparams}. They are divided into three sections of columns which represent properties of
\begin{itemize}
	\item the entire GFGs containing EACs --- section $F_n$,
	\item the condensation graphs of the related GFGs --- section $\mathrm{cg}(F_n)$,
	\item the EACs contained in the GFGs --- section $\mathrm{eac}(F_n)$.
\end{itemize}
The notation $\mathrm{eac}(F_n)$ is unambiguous in the context of the data in Table~\ref{tab:GFGparams}, because no GFG containing more than one EAC has been found so far.

In each section, the columns $v$ and $a$ include the number of vertices and arcs, respectively. In the sections $F_n$ and $\mathrm{cg}(F_n)$, the column $l$ represent the length of the longest directed path(s), expressed by the number of vertices belonging to the path(s). The numbers of
\begin{itemize}
	\item all connected components (column $c$),
	\item connected GACs, i.e. the GACs which are not disconnected components of a related $F_n$ (column $\overline{g}$),
	\item disconnected GACs, i.e. the GACs which are disconnected components of a related $F_n$ (column $\overset{\circ}{g}$)
\end{itemize}
are specified for GFGs. Additionally, the following parameters are given for EACs
\begin{itemize}
	\item the number of Hamiltonian paths (column $\overline{h}$),
	\item the number of Hamiltonian cycles (column $\overset{\circ}{h}$),
	\item the crossing number or its upper bound (column $x$).
\end{itemize}

\begin{table}
\begin{threeparttable}
\caption{Parameters of the found EACs and the GFGs which contain them}
\label{tab:GFGparams}
{\begin{tabular}{*{15}{r}}
\toprule
\ccol{$n$} & \multicolumn{6}{c}{$F_n$} & \multicolumn{3}{c}{$\mathrm{cg}(F_n)$} & \multicolumn{5}{c}{$\mathrm{eac}(F_n)$}\\
\cmidrule(lr){2-7} \cmidrule(lr){8-10} \cmidrule(lr){11-15}
& \ccol{$v$} & \ccol{$a$} & \ccol{$c$} & \ccol{$\overline{g}$} & \ccol{$\overset{\circ}{g}$} %
& \ccol{$l$} & \ccol{$v$} & \ccol{$a$} & \ccol{$l$} & \ccol{$v$} & \ccol{$a$} %
& \ccol{$\overline{h}$} & \ccol{$\overset{\circ}{h}$} & \ccol{$x$}\\
\midrule
128 & 30 & 49 & 3 & 1 & 2 & 11 & 20 & 19 & 4 & 8 & 15 & 5 & 0 & 0\\
1718 & 267 & 548 & 13 & 9 & 12 & 34 & 220 & 308 & 8 & 28 & 64 & 30 & 0 & $23^*$\\
1862 & 283 & 540 & 20 & 15 & 19 & 28 & 224 & 312 & 7 & 21 & 48 & 6 & 0 & $6^*$\\
1928 & 293 & 598 & 19 & 12 & 18 & 24 & 248 & 373 & 11 & 14 & 35 & 89 & 3 & $5^*$\\
2200 & 327 & 595 & 23 & 24 & 22 & 17 & 276 & 471 & 12 & 2 & 2 & 2 & 1 & 0\\
6142 & 800 & 1732 & 39 & 31 & 38 & 18 & 716 & 1258 & 10 & 10 & 27 & 12 & 0 & 1\\
\bottomrule
\end{tabular}}
\begin{tablenotes}
	\item $^*$ an upper bound
\end{tablenotes}
\end{threeparttable}
\end{table}

Exemplary longest directed paths in GFGs, having the length specified in Table~\ref{tab:GFGparams}, are as follows:
\begin{itemize}
	\item $F_{128} \longrightarrow$ (\ud{29}, \ud{41}, \ud{5}, \ud{3}, \ud{23}, \ud{13}, \ud{11}, \ud{7}, 37, 17, 43),
	\item $F_{1718} \longrightarrow$ (\ud{17}, \ud{137}, \ud{211}, \ud{241}, \ud{31}, \ud{571}, \ud{5}, \ud{13}, \ud{67}, \ud{43}, \ud{41}, \ud{37}, \ud{53}, \ud{181}, \ud{89}, \ud{383}, \ud{569}, \ud{11}, \ud{563}, \ud{29}, \ud{239}, \ud{523}, \ud{149}, \ud{79}, \ud{59}, \ud{7}, \ud{3}, 179, 107, 113, 23, 131, 277, 887),
	\item $F_{1862} \longrightarrow$ (\ud{167}, \ud{359}, \ud{67}, \ud{53}, \ud{113}, \ud{619}, \ud{5}, \ud{47}, \ud{593}, \ud{83}, \ud{617}, \ud{11}, \ud{179}, \ud{251}, \ud{607}, \ud{41}, \ud{17}, \ud{43}, \ud{13}, \ud{3}, 29, \ud{151}, \ud{503}, \ud{353}, \ud{97}, \ud{601}, \ud{59}, 269),
	\item $F_{1928} \longrightarrow$ (\ud{641}, \ud{5}, \ud{53}, \ud{73}, \ud{103}, \ud{383}, \ud{13}, \ud{17}, \ud{619}, \ud{71}, \ud{11}, \ud{3}, \ud{113}, \ud{7}, 31, 37, 263, 613, 89, 593, 149, 587, 167, 1093),
	\item $F_{2200} \longrightarrow$ (\ud{13}, \ud{3}, 727, 19, 281, \ud{233}, \ud{103}, \ud{37}, \ud{17}, \ud{7}, 23, 107, 167, 29, 83, 457, 829),
	\item $F_{6142} \longrightarrow$ (\ud{409}, \ud{7}, \ud{227}, \ud{13}, \ud{877}, \ud{3}, \ud{157}, \ud{19}, \ud{5}, 467, 71, 107, 43, 79, 59, 1009, 1097, 2851).
\end{itemize}
The given longest paths are not necessarily unique. The underscored numbers represent vertices inducing subpaths inside multi-vertex SCCs. All the paths start from EACs and traverse all ($F_{128}$, $F_{1928}$, $F_{2200}$) or all but one ($F_{1718}$, $F_{1862}$, $F_{6142}$) vertices of the contained EAC. In the case of $F_{1862}$ and $F_{2200}$, the longest paths also traverse all vertices of the internal multi-vertex SCCs, which have 6 and 5 vertices, respectively, according to the data presented in Figure~\ref{fig:GFGmaps}. Hence, as one could expect, larger SCCs are the parts of GFGs which provide long subpaths belonging to the longest paths.

All the found EACs have Hamiltonian paths and $\mathrm{eac}(F_{2200})$ has the trivial Hamiltonian cycle $(3,13)$, as it is induced by two vertices. Only $\mathrm{eac}(F_{1928})$ has the following non-trivial Hamiltonian cycles:
\begin{itemize}
	\item (641, 5, 113, 7, 3, 53, 73, 103, 383, 13, 17, 619, 71, 11),
	\item (641, 5, 113, 7, 17, 619, 71, 11, 3, 53, 73, 103, 383, 13),
	\item (641, 5, 53, 73, 103, 383, 13, 17, 619, 71, 11, 113, 7, 3).
\end{itemize}

The crossing numbers provided in the last column of Table~\ref{tab:GFGparams} are an indirect result of the effort to obtain neat and legible drawings of the found EACs. This effort includes the minimization of the number of arc crosses.

\subsection{Reliability of the numerical results}
The structure of GFGs is precisely described by Definition~\ref{def:GFG}, hence, for any $n \in \mathbb{N}_\mathrm{G}$, the graph $F_n$ can be built by a well defined algorithm, based on this definition. The construction of a condensation graph $\mathrm{cg}(G)$ of an arbitrary directed graph $G$ is the well known and tractable procedure, based on splitting $G$ into SCCs%
~\cite{Cormen2009}, which can be used to obtain $\mathrm{cg}(F_n)$ from $F_n$. Having $\mathrm{cg}(F_n)$ and, consequently, the set of SCCs of $F_n$, one can easily find EACs by iteration over all the SCCs and excluding the ones which: have predecessors in $\mathrm{cg}(F_n)$ (non-source SCCs), consist of only single vertex $p \in V_n$ (TACs or GACs), consist of two vertices $p,q \in V_n$ such that $p+q=n$ (GACs). The remaining SCCs are EACs. Thus, we have defined the algorithmic ways to obtain the data structures representing $F_n$, $\mathrm{cg}(F_n)$, and $\mathrm{eac}(F_n)$. These data structures contain easily accessible information about the number of vertices and arcs of the related graphs, hence, we obtain the values for the columns $v$ and $a$ of all the three sections of Table~\ref{tab:GFGparams}. While the SCCs of $F_n$ are iterated, GACs can be simultaneously counted and split into the ones which have successors in $\mathrm{cg}(F_n)$ (connected) and the others (disconnected). Thus, we get the values of $\overline{g}$ and $\overset{\circ}{g}$ from the section $F_n$ of Table~\ref{tab:GFGparams}.

To find the longest paths and Hamiltonian paths or cycles, optimization technique based on constraint programming has been applied. We assume that the formulas~(\ref{eq:PathsDecVar})-(\ref{eq:PathsMaxFcnWeight}) concern a directed graph $G=(V,A)$, which can be a GFG or an EAC (ignoring arc weights), where $V \subset \mathbb{N}_1$, $A \subseteq V^2$, and $\nu = \#V$. All the programs for search of paths or cycles use the decision variables
\begin{equation}
	x_i \in V, \quad \forall i \in \left\{1,2,\ldots,\nu\right\},
\label{eq:PathsDecVar}
\end{equation}
with the set of constraints of the type ''all different''
\begin{equation}
	i \neq j \implies x_i \neq x_j, \quad \forall (i,j) \in \left\{1,2,\ldots,v\right\}^2.
\label{eq:PathsAllDiff}
\end{equation}
Combining (\ref{eq:PathsDecVar}) and (\ref{eq:PathsAllDiff}) with the following constraints
\begin{equation}
	\exists_{(s,t) \in A} (s,t)=(x_i,x_{i+1}), \quad \forall i \in \left\{1,2,\ldots,\nu-1\right\}.
\label{eq:PathsHam}
\end{equation}
we obtain the program for search Hamiltonian paths. To get Hamiltonian cycles, the following additional constraints are needed
\begin{equation}
	\exists_{(s,t) \in A} (s,t)=(x_\nu,x_1),
\label{eq:PathsClose}
\end{equation}
\begin{equation}
	x_1 = x^*, \quad x^* \in V, \quad \text{$x^*$ is selected arbitrarily}.
\label{eq:PathsSymBrk}
\end{equation}
The constraint (\ref{eq:PathsClose}) represents the requirement that a found path has to be closed. The constraint~(\ref{eq:PathsSymBrk}) is provided for symmetry breaking in the problem solving process.

It is reasonable to assume that an arbitrary complete GFG may have no Hamiltonian path. To find a general, possibly non-Hamiltonian, longest path in a directed graph, we can use the program \texttt{LongestPathGen-P} consisting of~(\ref{eq:PathsDecVar}) and~(\ref{eq:PathsAllDiff}), and the following formulas
\begin{equation}
	c_1 \longleftarrow \lambda\left(\exists_{(s,t) \in A} (s,t)=(x_1,x_2)\right),
\label{eq:PathsC1}
\end{equation}
\begin{equation}
	c_i \longleftarrow \lambda\left(c_{i-1}=1 \land \exists_{(s,t) \in A} (s,t)=(x_i,x_{i+1})\right), %
	\;\forall i \in \left\{2,3,\ldots,\nu-1\right\},
\label{eq:PathsCi}
\end{equation}
\begin{equation}
	\eta=\sum\limits_{i=1}^{\nu-1}c_i, \quad \eta \rightarrow \max!,
\label{eq:PathsMaxFcn}
\end{equation}
where $\lambda$ is the function which evaluates to 1 if the statement provided as its argument is true, and it evalueates to 0 otherwise. The notation $\alpha \longleftarrow \beta$ used in~(\ref{eq:PathsC1}) and~(\ref{eq:PathsCi}) means that a decision expression $\beta$ is assigned to a symbol $\alpha$, hence, $\alpha$ is not an independent decision variable, but it is an alias of $\beta$. The used constraint programming solver supports statements semantically equivalent to the function $\lambda$ and the notation $\alpha \longleftarrow \beta$. According to the formulas~(\ref{eq:PathsC1})--(\ref{eq:PathsMaxFcn}), the program maximizes the value $\eta$ such that $(x_i,x_{i+1}) \in A$ for each $i \in \{1,2,\ldots,\eta\}$. It follows that, after the program finds an optimal solution, the longest path has the length equal to $\eta+1$ and it is represented by the sequence of vertices $\left(x_i\right)_{i=1}^{\eta+1}$.

In practice, the program \texttt{LongestPathGen-P} was not able to solve the problems optimally for the found GFGs containing EACs in a reasonable time, except for the case of $F_{128}$. Because of that, a more specialized program, \texttt{LongestPathGFG-P}, has been implemented, which makes use of the fact that the analysed GFGs contain mainly trivial SCCs, i.e. SCCs induced by single vertices.
\begin{figure}
	\centering
    \includegraphics[width=0.87\textwidth,trim={16mm 66mm 68mm 10mm},clip]{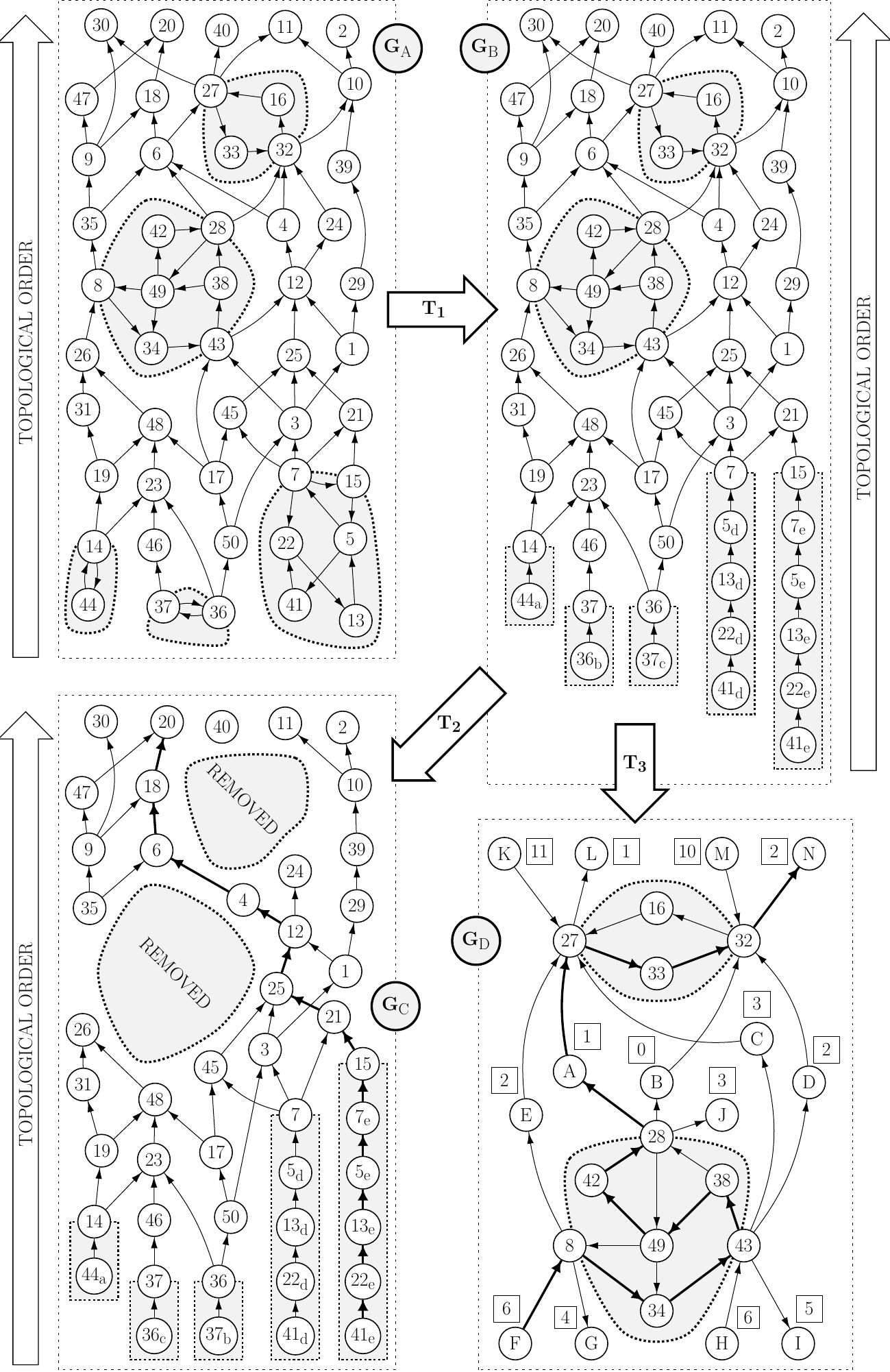}
	\caption{The concept of the program \texttt{LongestPathGFG-P}}
	\label{fig:LongPathAlg}
\end{figure}
An example in Figure~\ref{fig:LongPathAlg} presents the main phases of the program \texttt{LongestPathGFG-P} execution. The example is based on a hypothetical but quintessential graph model which has one larger source SCC (a counterpart of an EAC), a few small source SCCs (counterparts of GACs), and a few inner multi-vertex SCCs (analogously as in $F_{1862}$, $F_{1928}$, $F_{2200}$, and $F_{6142}$). First, using standard algorithms, SCCs are extracted in the given graph, its condensation is derived and sorted topologically, which results in $G_\mathrm{A}$ from Figure~\ref{fig:LongPathAlg}. If the longest path in $G_\mathrm{A}$ goes through a source SCC, it has to include the longest subpath inside this SCC, as the entire path would not be the longest one otherwise. Therefore, in the transformation $\mathrm{T}_1$, the possible longest subpaths inside source SCCs are found and substituted by disjunctive simple paths of the same length, constructed with the vertices labelled by the numbers with the additional subscripts a--e in $G_\mathrm{B}$. In this way, the longest paths in $G_\mathrm{A}$ and $G_\mathrm{B}$ have the same length, but the cycles inside source SCCs are eliminated in $G_\mathrm{B}$. The longest paths inside these SCCs can be found using the procedure based on the formulas (\ref{eq:PathsDecVar}), (\ref{eq:PathsAllDiff}), (\ref{eq:PathsC1})--(\ref{eq:PathsMaxFcn}), with the additional constraint $x_1=v^*$, which fixes the first vertex of the path at $v^* \in V_n$. In the considered case, the last vertex of the subpath is fixed, therefore, arc directions need to be reverted before and after the procedure execution. In the transformation $\mathrm{T}_2$, the vertices inducing inner multi-vertex SCCs (MVSCCs) are removed, so $G_\mathrm{C}$ is acyclic and contains the longest path $L_\mathrm{C}$ bypassing the inner MVSCCs, which can be found with a negligible computational burden. The transformation $\mathrm{T}_3$, in turn, results in the graph $G_\mathrm{D}$ prepared for search the longest path $L_\mathrm{D}$ which includes at least one vertex inducing an inner MVSCC. Because of its maximality, the path $L_\mathrm{D}$ includes the longest subpaths to, from, and between the vertices belonging to inner MVSCCs. The transformation $\mathrm{T}_3$ finds these longest subpaths and contracts them to single weighted vertices with the weights equal to the length of the contracted subpaths. The weights are disclosed in the rectangles next to vertices in Figure~\ref{fig:LongPathAlg}. The transformation $\mathrm{T}_3$ requires insignificant computational effort, because it finds the longest subpaths in an acyclic subgraph of $G_\mathrm{B}$, but it significantly decreases the number of vertices in the target cyclic graph $G_\mathrm{D}$ and makes its processing, which is an NP-hard problem, easier. To find the longest path in $G_\mathrm{D}$, a modified variant of the program \texttt{LongestPathGen-P} can be used which takes into account the weights of vertices. The modification is required in the definition of the objective function~(\ref{eq:PathsMaxFcn}), which takes the form
\begin{equation}
	\eta=\sum\limits_{i=1}^{\nu-1}w[x_i]c_i, \quad \eta \rightarrow \max!,
\label{eq:PathsMaxFcnWeight}
\end{equation}
where $w[x_i]$ denotes the weight of the vertex assigned to the decision variable $x_i$. The used constraint programming solver supports statements semantically equivalent to the notation $w[x_i]$. After the paths $L_\mathrm{C}$ and $L_\mathrm{D}$ are determined, the longer of them is the longest path in the original input graph. The program \texttt{LongestPathGFG-P} has found the longest paths in all the graphs listed in Table~\ref{tab:GFGparams} in a reasonably short time. All the embedded routines of the \texttt{LongestPathGFG-P} executed in the form of constraint programming tasks finished with the proved optimality of a result.

It is proved that determining the crossing number is an NP-hard problem%
~\cite{Garey1983}, even in the case of cubic graphs%
~\cite{Hlineny2006}. However, the graph $\mathrm{eac}(F_{2200})$ has a trivial structure (Figure~\ref{fig:EAC2200and6142}a) and it is obviously planar. Planarization of $\mathrm{eac}(F_{128})$ is easy and one can obtain its planar drawing without computer aid, one such possible drawing is presented in Figure~\ref{fig:GFG128}. It has been checked using the Boyer-Myrvold algorithm%
~\cite{Boyer2004} that the remaining found EACs are non-planar. To find upper bounds on the crossing numbers of these non-planar EACs, a dedicated approximate algorithm has been prepared. This algorithm uses a model of a graph drawing in which vertices are arranged in nodes of a finite two-dimensional grid and arcs are represented by line segments. A discrete optimization metaheuristic is employed in the algorithm to search for such vertex arrangement which minimizes the number of arc crosses. The minimized numbers obtained for the graphs $F_{1718}$, $F_{1862}$, $F_{1928}$, and $F_{6142}$ are given in the last column of Table~\ref{tab:GFGparams}, as the upper bounds on the crossing numbers. The arcs in the respective drawings have been bent to obtain the legible illustrations (Figures \ref{fig:EAC1718}, \ref{fig:EAC1862}, \ref{fig:EAC1928}, \ref{fig:EAC2200and6142}b), but it has not changed the number of arc crosses. The graph $\mathrm{eac}(F_{6142})$ is non-planar, but its crossing number is not greater than 1, thus, it is exactly equal to 1.

In conclusion, most of the numerical results given in Table~\ref{tab:GFGparams} are reliable in this sense that one can replicate them using well defined input data and the described deterministic, tractable algorithms. In the case of a few parameters, NP-hard problems have to be solved to obtain their values. For this purpose, the programs using constraint programming technique have been developed, which provide optimal results in a practically short time of execution on a modern computer. The used constraint programming solver has not been verified separately and the reliability of the related results relies on the trust that it works correctly. The values of the crossing numbers or their upper bounds are proven directly by Figures~\ref{fig:EAC1718}--\ref{fig:EAC2200and6142}.

\section{Final remarks}
\label{sec:Remarks}
On the basis of the achieved results, we can answer the questions stated in Introduction.
\begin{enumerate}
	\item\emph{Is it possible to obtain rigorous statements related to the EACs existence, based on the existing mathematical knowledge?}

It turned out to be possible in the case of twin EACs. For a general case, no explicit results have been found and the problem remains open for future research.

	\item\emph{How to search for EACs with the use of computer-aided techniques and what are the results of such experimental research?}

The relevant algorithm has been proposed, implemented, and used for computational experiments. Essential results of this work have been obtained from the experimental study based on this algorithm.

	\item\emph{Can we characterize the cardinality of the set of EACs: empty, finite, infinite?}

The set of EACs is non-empty. It is a symptomatic observation in the context of the belief, that there are no counterexamples to the BGC. The research has not determined finitude of the set of EACs, but the results provide a suggestion, that it can be finite.

	\item\emph{Are there some general properties of EACs?}

The performed research has not revealed such clear, explicit properties. The found EACs differ in many ways, e.g., the number of vertices, existence of Hamiltonian cycles, planarity. On the other hand, the properties for the comparison have been chosen rather arbitrarily, hence, some of them may differ and others may be common. The only obvious common property of the found EACs is that there exists at most one EAC per GFG. However, it can be a coincidence correlated with the fact that EACs are very sparsely distributed over GFGs. 

	\item\emph{If EACs exist, is it practical (regarding their number and sizes) to describe each of them individually?}
	
Only 6 EACs has been found and the largest one consists of 28 vertices. This made it practical to prepare a survey of the EACs, including presentation of basic properties of the GFGs containing them. This survey opens a kind of repository of these interesting mathematical objects, which can be potentially extended in the future if new EACs are found.
\end{enumerate}

In the context of the obtained results, it is legitimate to conclude with the following conjecture.
\begin{conjecture}
There exist exactly six exceptional autonomous components of Goldbach factorization graphs.
\label{con:EAC}
\end{conjecture}
Conjecture~\ref{con:EAC} is based merely on experimental results and cannot be considered as very strong. Nevertheless, this conjecture indicates a few directions of further research:
\begin{enumerate}
	\item According to Theorem~\ref{th:EACGoldbach} and the verified properties of the GFGs containing EACs, Conjecture~\ref{con:EAC} implies the BGC. It is, therefore, reasonable to try to prove Conjecture~\ref{con:EAC}. It is hard to predict it would be an efficient way to attack the BGC, but it cannot be excluded, and the case of twin EACs is a successful example.
	\item The another option is to study entire GFGs in more details. It is characteristic that no $n \in \mathbb{N}_\mathrm{G}$, such that $F_n$ contains EACs, contradicts the BGC. Maybe there exists a specific structural property of GFGs which enforces the existence of GACs (i.e.~Goldbach partitions) in any GFG, regardless of the existence of EACs.
	\item Conjecture~\ref{con:EAC} is stronger than the BGC. It is, therefore, realistic that the second mentioned conjecture is proved and the first one will not. In this sense, the search of EACs can be considered as a problem unrelated to the BGC and important in and of itself, because EACs represent interesting self-conjugation of primes under a relation which combines addition and multiplication.
\end{enumerate}

A basic, obvious and useful direction of further research is to continue the computer-aided search of EACs. It may seem to be an easy task, which could be performed with the use of the existing algorithm. However, the finished search up to $n=10^8$ has consumed most of available memory and large amount of work time (a few months) of a single modern personal computer. The continuation of the search is, therefore, a challenging problem involving large-scale distributed computing and, possibly, significant improvement of the algorithm.

\bibliographystyle{amsplain}
\bibliography{bibdata}

\providecommand{\bysame}{\leavevmode\hbox to3em{\hrulefill}\thinspace}
\providecommand{\MR}{\relax\ifhmode\unskip\space\fi MR }
\providecommand{\MRhref}[2]{%
  \href{http://www.ams.org/mathscinet-getitem?mr=#1}{#2}
}
\providecommand{\href}[2]{#2}
\begin{thebibliography}{10}

\bibitem{Bennett2001}
M.~A. Bennett, \emph{{O}n {S}ome {E}xponential {E}quations of {S}. {S}.
  {P}illai}, Canad. J. Math. \textbf{53} (2001), no.~5, 897--922.

\bibitem{Boyer2004}
J.~M. Boyer and W.~J. Myrvold, \emph{On the {C}utting {E}dge: {S}implified
  ${O}(n)$ {P}lanarity by {E}dge {A}ddition}, Journal of Graph Algorithms and
  Applications \textbf{8} (2004), no.~3, 241--273.

\bibitem{Chen1973}
J.~R. Chen, \emph{On the representation of a larger even integer as the sum of
  a prime and the product of at most two primes}, Sci. Sinica \textbf{16}
  (1973), 157--176.

\bibitem{Cormen2009}
T.~H. Cormen, C.~E. Leiserson, R.~L. Rivest, and C.~Stein, \emph{Introduction
  to {A}lgorithms}, third ed., The MIT Press, 2009.

\bibitem{Deshouillers1997}
J.-M. Deshouillers, G.~Effinger, H.~{te Riele}, and D.~Zinoviev, \emph{A
  complete {V}inogradov 3-primes theorem under the {R}iemann hypothesis},
  Electronic Research Announcements \textbf{3} (1997), 99--104.

\bibitem{Deshouillers1993}
J.-M. Deshouillers, A.~Granville, W.~Narkiewicz, and C.~Pomerance, \emph{An
  {U}pper {B}ound in {G}oldbach's {P}roblem}, Mathematics of Computation
  \textbf{61} (1993), no.~20, 2009--2013.

\bibitem{Estermann1938}
T.~Estermann, \emph{On {G}oldbach's {P}roblem : {P}roof that {A}lmost all
  {E}ven {P}ositive {I}ntegers are {S}ums of {T}wo {P}rimes}, Proceedings of
  the London Mathematical Society \textbf{s2-44} (1938), no.~1, 307--314.

\bibitem{Fliegel1989}
H.~F. Fliegel and D.~S. Robertson, \emph{{Goldbach’s Comet: the numbers
  related toGoldbach’s Conjecture}}, Journal of Recreational Mathematics,
  \textbf{21} (1989), no.~1, 1--7.

\bibitem{Garey1983}
M.~R. Garey and D.~S. Johnson, \emph{Crossing {N}umber is {NP}-{C}omplete},
  SIAM Journal on Algebraic Discrete Methods \textbf{4} (1983), no.~3,
  312--316.

\bibitem{Gauss1986}
C.~F. Gauss, \emph{Disquisitiones {A}rithmeticae. {T}ranslated by {A}rthur {A}.
  {C}larke}, Springer, New York, 1986.

\bibitem{Hardy1924}
G.~H. Hardy and J.~E. Littlewood, \emph{{Some problems of Partito Numerorum V.
  A further contribution to the study of Goldbach’s problem}}, Proceedings of
  the London Mathematical Society \textbf{s2-22} (1924), no.~1, 46--56.

\bibitem{Hardy1918}
G.~H. Hardy and S.~Ramanujan, \emph{Asymptotic formulae in combinatory
  analysis}, Proceedings of the London Mathematical Society \textbf{s2-17}
  (1918), no.~1, 75--115.

\bibitem{Helfgott2013}
H.~A. Helfgott, \emph{The ternary {G}oldbach conjecture is true}, arXiv
  e-prints (2013), arXiv:1312.7748.

\bibitem{Hlineny2006}
P.~Hliněný, \emph{Crossing number is hard for cubic graph}, Journal of
  Combinatorial Theory, Series B \textbf{96} (2006), no.~4, 455--471.

\bibitem{Luca2003}
F.~Luca, \emph{On the diophantine equation $p^{x_1}-p^{x_2}=q^{y_1}-q^{y_2}$},
  Indagationes Mathematicae \textbf{14} (2003), no.~2, 207--222.

\bibitem{Mignotte1999}
M.~Mignotte and A.~Pethő, \emph{{O}n the {D}iophantine {E}quation
  $x^p-x=y^q-y$}, Publicacions Matem\`atiques \textbf{43} (1999), 207--216.

\bibitem{Silva2014}
T.~{Oliveira e Silva}, S.~Herzog, and S.~Pardi, \emph{Empirical verification of
  the even {G}oldbach conjecture and computation of prime gaps up to $4 \cdot
  10^{18}$}, Mathematics of Computation \textbf{83} (2014), no.~288,
  2033--2060.

\bibitem{Pillai1936}
S.~S. Pillai, \emph{On $a^x+b^y=c$}, Indian Math. Soc. (N.S.) \textbf{2}
  (1936), 119--122.

\bibitem{Pillai1937}
\bysame, \emph{A correction to the paper {O}n $a^x+b^y=c$}, Indian Math. Soc.
  (N.S.) \textbf{2} (1937), 215.

\bibitem{Ramanujan1919}
S.~Ramanujan, \emph{A proof of {B}ertrand's postulate}, Journal of the Indian
  Mathematical Society \textbf{11} (1919), 181--181.

\bibitem{Scott1993}
R.~Scott, \emph{On the {E}quations $p^x-b^y=c$ and $a^x+b^y=c^z$}, Journal of
  Number Theory \textbf{44} (1993), 153--165.

\bibitem{Shen1964}
M.-K. Shen, \emph{On checking the {G}oldbach conjecture}, BIT \textbf{4}
  (1964), no.~4, 243--245.

\bibitem{Terai1999}
N.~Terai, \emph{Applications of a lower bound for linear forms in two
  logarithms to exponential {D}iophantine equations}, Acta Arithmetica
  \textbf{90} (1999), 17--35.

\bibitem{Yamada2015}
T.~Yamada, \emph{Explicit {C}hen's theorem}, arXiv e-prints (2015),
  arXiv:1511.03409.

\end{thebibliography}

\end{document}